\documentclass{article}
\usepackage{amsmath}
\usepackage{amssymb}
\usepackage{amsfonts,color}
\usepackage{dsfont}

\newtheorem{theorem}{Theorem}[section]

\newtheorem{corollary}[theorem]{Corollary}

\newtheorem{example}[theorem]{Example}

\newtheorem{lemma}[theorem]{Lemma}

\newtheorem{proposition}[theorem]{Proposition}

\newenvironment{proof}[1][Proof]{\noindent\textbf{#1.} }{\ \rule{0.5em}{0.5em}}

\begin{document}

\title{Hitting Probability and the Hausdorff Measure of the Level sets for
Spherical Gaussian Fields}
\author{Xiaohong Lan \\
School of Mathematical Sciences \\
University of Science and Technology of China \\
E-mail: xhlan@ustc.edu.cn}
\maketitle

\begin{abstract}
Consider an isotropic spherical Gaussian random field $\mathbf{T}$ with
values in $%
\mathbb{R}
^{d}$. We investigate two problems: (i) When is the level set $\mathbf{T}%
^{-1}(\mathbf{t})\ $nonempty with positive probabibility for any $\mathbf{%
t\in }%
\mathbb{R}
^{d}$ ? (ii) If the level set is nonempty, what is its Hausdorff measure?
These two question are not only very important in potential theory for
random fields, but also foundamental in geometric measure theory. We give a
complete answer to the questions under some very mild conditions.
\end{abstract}



\section{Introduction and Statement of Main Results}


The investigation of geometric properties of level sets of random fields is
an interesting topic in modern probability and in several applications.
Depending on whether the sample functions of the random field are smooth
(e.g. continuously differentiable) or not, the properties of the level sets
have to be studied using different geometric or topological tools. We refer
to Adler \cite{Adler81}, Adler and Taylor \cite{AdlerTaylor, AdlerT11},
Adler et al. \cite{AdlerTW12}, and Aza\"is and Wschebor \cite{AW09} for
systematic accounts on geometry of random fields and applications.

One of our interest is to find the connection between the existence of
nonempty level sets and capacity, or we say the hitting probability $\mathbb{%
P}\left\{ \mathbf{T}^{-1}(\mathbf{t})\neq \varnothing \right\} $ for any $%
\mathbf{t\in }%
\mathbb{R}
^{d}$ in our case. The problem has been studied widely for $d=1$ with $1$%
-dimensional random processes and we cite a few such as \cite{Bert} and \cite%
{Hawkes}, etc. For $d>1$ and multi-dimensional random fields on Euclidean
space, it has been only known for a few, see for instance \cite{BierLacXiao}%
, \cite{KhoshShi} and \cite{KhoshXiao}. Unfortunately, for $d>1$ and
multi-dimensional random fields on non-Euclidean space, there is no known
result on this problem. We show that when some conditions hold, the hitting
probability is indeed positive.

Whence proved the existence of nonempty level sets under some circumstances,
a naturally question is, what is the volume of this level set? When the
sample functions satisfy certain smoothness/differentiability conditions,
the expected value of the volume of these level sets can be computed
explicitly in a wide variety of circumstances, by exploiting techniques
based on the Kac-Rice formula and its generalizations (see \cite{AdlerTaylor}%
, Chapters 11-12). For instance, the length of the level curve for
stationary Gaussian random fields on the two-dimensional plane was
investigated by Kratz and Le\'{o}n \cite{KratzLeon, KratzLeon06, KratzLeon10}%
.

The applicability of the Kac-Rice approach is restricted to cases where
differentiability conditions are ensured. When the sample functions of the
random field are nowhere differentiable, the level sets are often fractals
and their geometric properties are closely related to the analytic
properties of the local times of the random field. See Geman and Horowitz 
\cite{GeHor}, Xiao \cite{Xiao97,Xiao09}, and the references therein for more
information. In this paper, we extend the methods for studying the Hausdorff
measure of level sets for Gaussian fields to the spherical setting. The new
main ingredient is the strong local nondeterminism and upper bound of higher
order moment of local time established recently in \cite{LanMarXiao} and 
\cite{LanXiao}, respectively.

More precisely, we shall focus on the level sets of an isotropic Gaussian
random field $\mathbf{T}=\left\{ \mathbf{T}(x),\,x\in \mathbb{S}^{2}\right\} 
$ with values in $\mathbb{R}^{d}$ defined on some probability space $(\Omega
,\Im ,\mathbb{P})$ by 
\begin{equation}
\mathbf{T}(x)=\left( T_{1}(x),\ldots ,T_{d}(x)\right) ,\quad x\in \mathbb{S}%
^{2},  \label{Def:T}
\end{equation}%
where $T_{1},\ldots ,T_{d}$ are independent copies of $T_{0}=\left\{
T_{0}(x),\quad x\in \mathbb{S}^{2}\right\} $. We assume $T_{0}$ to be a
zero-mean, mean square continuous and isotropic random fields on the sphere,
for which the following spectral representation holds (\cite{MPbook},
Chapter 5) : for $x\in \mathbb{S}^{2},$ 
\begin{equation}
T_{0}(x)=\sum_{\ell \geq 0}\sum_{m=-\ell }^{\ell }a_{\ell m}Y_{\ell m}(x)
\label{rapT}
\end{equation}%
where $\{a_{\ell m}\}_{\ell ,m},\ \ell =0,1,2,...,$ $m=-\ell ,...,\ell $ is
a triangular array of zero-mean, orthogonal, complex-valued random variables
with variance $\mathbb{E}|a_{\ell m}|^{2}=C_{\ell },$ the angular power
spectrum of the random field. For $m<0$ we have $a_{lm}=(-1)^{m}\overline{%
a_{l-m}}$ , whereas $a_{l0}$ is real with the same mean and variance. (\ref%
{rapT}) holds both in $L^{2}(\Omega )$ at every fixed $x,$ and in $%
L^{2}(\Omega \times {\mathbb{S}}^{2}),$ i.e.%
\begin{equation}
\lim_{L\rightarrow \infty }\mathbb{E}\bigg[T(x)-\sum_{\ell
}^{L}\sum_{m}a_{\ell m}(\omega )Y_{\ell m}(x)\bigg]^{2}=0,  \label{rapT-MSC}
\end{equation}%
and 
\begin{equation}
\lim_{L\rightarrow \infty }\mathbb{E}\left\vert \int_{S^{2}}\left\{
T_{0}(x)-\sum_{l=1}^{L}\sum_{m=-l}^{l}a_{lm}Y_{lm}(x)\right\} ^{2}d\nu
(x)\right\vert ^{2}=0\text{ ,}  \label{rapT-L2}
\end{equation}%
Here $\nu $ denotes the canonical Lebesgue measure on the unit sphere with $%
d\nu (x)=\sin \vartheta d\vartheta d\varphi $ in the spherical coordinates $%
(\vartheta ,\varphi )\in \lbrack 0,\pi ]\times \lbrack 0,2\pi )$. The
functions $\left\{ Y_{\ell m}(x)\right\} $ are the so-called spherical
harmonics, i.e. the eigenvectors of the Laplacian operator on the sphere. It
is a well-known result that the spherical harmonics provide a complete
orthonormal systems for $L^{2}(S^{2}),$ see again \cite{MPbook}. A
celebrated theorem of Schoenberg \cite{schoenberg1942} provides the
following expansion for the covariance function: 
\begin{equation}
\mathbb{E}\left[ T_{0}(x)T_{0}(y)\right] =\sum_{\ell =0}^{+\infty }C_{\ell }%
\frac{2\ell +1}{4\pi }P_{\ell }\left( \langle x,y\rangle \right) ,
\label{CovT0}
\end{equation}%
where for $\ell =0,1,2,...,$ $P_{\ell }:[-1,1]\rightarrow \mathbb{R}$
denotes the Legendre polynomial, which satisfy the normalization condition $%
P_{\ell }(1)=1$. Thus, without loss of generality, we assume for every $x\in 
\mathbb{S}^{2},$ 
\begin{equation*}
\mathbb{E}\left[ T_{0}\left( x\right) ^{2}\right] =\sum_{\ell =0}^{+\infty
}C_{\ell }\frac{2\ell +1}{4\pi }:=1.
\end{equation*}

We introduce now the same, mild regularity conditions on the angular power
spectrum $C_{\ell }$ of the random field $T_{0}(x)$ as in \cite{LanMarXiao}.

\ \textbf{Condition} (\textbf{A)} The random field $T_{0}( x) $ is Gaussian
and isotropic with angular power spectrum such that, for all $\ell \geq 1,$
there exist constants $\alpha >2,$ $K_{0}>1,$ such that 
\begin{equation}  \label{Specon}
C_{\ell }=\ell ^{-\alpha }G(\ell )>0,
\end{equation}
where 
\begin{equation*}
K_{0}^{-1}\leq G(\ell )\leq K_{0}\text{ for all }\ell \in \mathbb{N}^{+}%
\text{.}
\end{equation*}

Condition \textbf{(A)} entails a weak smoothness requirement on the tail
behavior of the angular power spectrum, which is trivially satisfied by some
cosmologically relevant models (where the angular power spectrum usually
behaves as an inverse polynomial, see \cite{Dodelson}, pp.243-244).

Now denote by 
\begin{equation*}
\mathbf{T}^{-1}(\mathbf{t})=\left\{ x\in \mathbb{S}^{2}:\mathbf{T}(x)=%
\mathbf{t}\right\}
\end{equation*}%
the level set at any $\mathbf{t}\in \mathbb{R}^{d}$ and $\phi $-$m$ the
Hausdorff measure associated to the function $\phi $. We state the following
result for the critical condition on the existence of nonempty level sets as
well as determining the exact Hausdorff measure function for the level set $%
\mathbf{T}^{-1}(\mathbf{t})$. In particular, it implies that the Hausdorff
dimension of $\mathbf{T}^{-1}(\mathbf{t})$ equals $2-(\alpha -2)d/2$ in view
of Theorem 1.2 in \cite{LanXiao}.

\begin{theorem}
\label{Th:Main} Let $\mathbf{T}=\left\{ \mathbf{T}(x),x\in \mathbb{S}%
^{2}\right\} $ be a Gaussian random field with values in $\mathbb{R}^{d}$
defined in $(\ref{Def:T})$. Assume that the associated isotropic random
field $T_{0}$ satisfies Condition (\textbf{A}) with $2<\alpha <4$. Then 
\begin{equation}
\mathbb{P}\left\{ \mathbf{T}^{-1}(\mathbf{0})\neq \varnothing \right\} >0\
\iff \ 4-\left( \alpha -2\right) d>0;  \label{Th:Existence}
\end{equation}%
Moreover, there exists a constant $K_{1}>1$ such that for every $\mathbf{t}%
\in \mathbb{R}^{d},$ the Hausdorff measure of the level set $\mathbf{T}^{-1}(%
\mathbf{t})$ associated with $\phi $ satisfies that 
\begin{equation}
K_{1}^{-1}L\left( \mathbf{t},\mathbb{S}^{2}\right) \leq \phi \text{-}m\left( 
\mathbf{T}^{-1}(\mathbf{t})\right) \leq K_{1}L\left( \mathbf{t},\mathbb{S}%
^{2}\right) ,\ a.s.\text{ .}  \label{Th:Measure}
\end{equation}%
where $L\left( \mathbf{t},\mathbb{S}^{2}\right) $ is the local time defined
in (\ref{defLT}) Section \ref{Sec:Preliminaries}, and the function $\phi $
is defined by 
\begin{equation}
\phi \left( r\right) =\frac{r^{2}}{\left[ \rho _{\alpha }(r/\sqrt{\log
\left\vert \log r\right\vert })\right] ^{d}}  \label{def:psi}
\end{equation}%
with $\rho _{\alpha }(r)=r^{\frac{\alpha }{2}-1}$ for $r\geq 0.$
\end{theorem}

In general, it has been known that Hausdorff measure of level sets can be
controlled lower bounded by relative local time, see for instance Xiao \cite%
{Xiao97,Xiao09} for Gaussian fields indexed in Euclidean space. Our theorem
above significantly improves their results by proving that the upper bound
by the local time also holds.

The rest of this paper is as follows\emph{: }we collect some technical
lemmas and their proofs in Section \ref{Sec:Preliminaries}. The critical
conditions on the existence of nonempty level set is presented in Section %
\ref{Sec:Capacity}. 
Section \ref{Sec:Measure} deals with the Hausdorff measure for the level
sets. The lower bound for the Hausdorff measure is derived by applying our
result on the local times and an upper density theorem of Roger and Taylor 
\cite{RogTaylor}. The proof of the upper bound is more difficult and we
extend the covering argument in Talagrand \cite{Talagrand95, Talagrand98}
and Xiao \cite{Xiao97}, and their strengthened version of Barak and
Mountford \cite{BarakMount} to the spherical setting. This part is technical
and is presented in Section \ref{Sec:PropSmooth}. As an illustration of
Theorem \ref{Th:Main}, we give two examples in Section \ref{Sec:Measure} as
well.

We denote by $\left\vert \mathbf{u}\right\vert
_{\infty}=\sup_{i=1,...,d}\left\vert u_{i}\right\vert ,$ the supremum of $%
\left\vert u_{i}\right\vert ,i=1,...,d,$ for $\mathbf{u}\in \mathbb{R}^{d},$
and use $K$ to denote a constant whose value may change in each appearance,
and $K_{i,j}$ to denote the $j$th more specific positive finite constant in
Section $i.$ 

\medskip

\textbf{Acknowledgement} The author thanks Professors Domenico Marinucci and
Yimin Xiao for stimulating discussions and helpful comments on this paper.
Research of X. Lan is supported by NSFC grant 11501538 and CAS grant
QYZDB-SSW-SYS009.

\section{Preliminaries \label{Sec:Preliminaries}}

In this section, we collect a few technical results which will be
instrumental for most of the proofs to follow. Recalling formula $\left( \ref%
{CovT0}\right) $, for any two points $x,y\in \mathbb{S}^{2},$ the variogram
of $T_{0}$ is defined by 
\begin{eqnarray}
\sigma ^{2}(x,y) &=&\mathbb{E}\left\{ \left[ T_{0}(x)-T_{0}(y)\right]
^{2}\right\}  \label{variogram-sig} \\
&=&2-2\mathbb{E}\left[ T_{0}(x)T_{0}(y)\right] =\sigma ^{2}\left( d_{\mathbb{%
S}^{2}}(x,y)\right)  \notag
\end{eqnarray}%
Here we have made a little abuse of the notation $\sigma ^{2}.$ The
following lemma from \cite{LanMarXiao} characterizes the variogram and the
property of strong local nondeterminism of $T_{0}$.

\begin{lemma}
\label{Lem:C1C2} Under Condition (\textbf{A}) with $2<\alpha <4$, there
exist positive constants $K_{2,1}\geq 1,$ $0<\delta _{0}<1$ depending only
on $\alpha $ and $K_{0},$ such that for any $x,y\in \mathbb{S}^{2},$ if $d_{%
\mathbb{S}^{2}}(x,y)<\delta _{0},$ we have 
\begin{equation}  \label{variogram}
K_{2,1}^{-1}\rho _{\alpha }^{2}( d_{\mathbb{S}^{2}}(x,y)) \leq \sigma ^{2}(
d_{\mathbb{S}^{2}}( x,y) ) \leq K_{2,1}\rho _{\alpha }^{2}( d_{\mathbb{S}%
^{2}}(x,y)) .
\end{equation}
Moreover, there exists a constant $K_{2,2}>0$ depending on $\alpha $ and $%
K_{0}$ only, such that for all integers $n\geq 1$ and all $%
x,x_{1},...,x_{n}\in \mathbb{S}^{2},$ we have 
\begin{equation}  \label{SLND}
\mathrm{Var}\left( T_{0}( x) |T_{0}( x_{1}),...,T_{0}( x_{n}) \right) \geq
K_{2,2}\min_{1\leq k\leq n}\rho_{\alpha }^{2}\left( d_{\mathbb{S}^{2}}(
x,x_{k}) \right) .
\end{equation}
\end{lemma}

\bigskip For any fixed point $x_{0}\in \mathbb{S}^{2},$ define the spherical
random field $Y_{x_{0}}(x)=T_{0}(x)-T_{0}(x_{0}),\ x\in \mathbb{S}^{2}.$ An
immediate consequence of Lemma \ref{Lem:C1C2} is as follows:

\begin{corollary}
\label{C2'} Under Condition (\textbf{A}) with $2<\alpha <4$, there exists a
constant $K_{2,2}^{\prime }>0,$ such that for all integers $n\geq 1$ and all 
$x,x_{1},...,x_{n}\in \mathbb{S}^{2},$

\begin{equation*}
\mathrm{Var}\left( Y_{x_{0}}(x)|Y_{x_{0}}(x_{1}),...,Y_{x_{0}}(x_{n})\right)
\geq K_{2,2}^{\prime }\min_{0\leq k\leq n}\left\{ \rho _{\alpha }^{2}(d_{%
\mathbb{S}^{2}}(x,x_{k}))\right\} .
\end{equation*}
\end{corollary}

Now let us introduce the local times of a random field on the unit sphere.
Recall first that, for any Borel sets $B\subset \mathbb{R}^{d}$, $D\subseteq 
\mathbb{S}^{2}$ and sample $\omega \in \Omega ,$ the \emph{occupation measure%
} of a spherical random field $\mathbf{T}$ in $B$ is defined by (c.f. \cite%
{GeHor}) 
\begin{equation*}
\mu _{D}\left( B,\omega \right) :=\int_{D}\mathrm{1}_{B}\left( \mathbf{T}%
\left( x,\omega \right) \right) d\nu \left( x\right) =\nu \left\{ x\in D:%
\mathbf{T}\left( x,\omega \right) \in B\right\} ,
\end{equation*}%
with $\mathrm{1}_{B}\left( \cdot \right) $ the index function. We recall
also that \emph{local times} of $\mathbf{T}$ exist on $\mathbb{R}^{d}$
provided that the measure $\mu _{D}(d\mathbf{t})$ is $a.s.$ absolutely
continuous \emph{w.r.t} the Lebesgue measure on $\mathbb{R}^{d}$. More
precisely, when $\mu $ is absolutely continuous, there exists a function $L(%
\mathbf{t},D)=L\left( \mathbf{t},D,\omega \right) \geq 0$ which is
measurable on $\mathbb{S}^{2}$ and such that for almost all $\omega \in
\Omega $ and each measurable set $B\subset \mathbb{R}^{d},$ 
\begin{equation}
\mu _{D}(B,\omega )=\int_{B}L(\mathbf{t},D,\omega )d\mathbf{t}\text{ }.
\label{defLT}
\end{equation}%
We call $L(\mathbf{t,}D)$ the \emph{local time} of $\mathbf{T}$ on $\mathbb{R%
}^{d}.$ See \cite{Pitt, Xiao09} for more information about local times of
random fields on the Euclidean spaces.

The following results are essential to our discussion about the inequalities 
$\left( \ref{Th:Measure}\right) $, and have been established in \cite%
{LanXiao}.

\begin{lemma}
\label{Lem:Local Time} Under the conditions of Theorem \ref{Th:Main} with $(
\alpha -2) d<4,$ $\mathbf{T}$ a.s. has a jointly continuous local time $L( 
\mathbf{t},D) $ \emph{w.r.t.} $( \mathbf{t,}x,r) $ for any open disk $D=D(
x,r) \subseteq \mathbb{S}^{2}$ and $\mathbf{t}\in \mathbb{R}^{d},$ which can
be represented in the $L^{2}(\mathbb{R}^{d})$-sense as 
\begin{equation}  \label{repLT}
L( \mathbf{t},D) =\frac{1}{2\pi }\int_{\mathbb{R}^{d}}e^{-i\boldsymbol{\xi }%
^{T}\mathbf{t}}\int_{D}e^{i\boldsymbol{\xi }^{T}\mathbf{T}( x) }d\nu ( x) d%
\boldsymbol{\xi };
\end{equation}
Moreover, there exists a positive constant $K_{2,3}$ depending on $\alpha ,$ 
$d,\ K_{0\text{ }}$ and $\gamma ,$ such that for any $\mathbf{s},\mathbf{t}%
\in \mathbb{R}^{d}$, all even integers $n\geq 2,\,0<\gamma <1$, we have 
\begin{equation*}
\mathbb{E}\left\{ \left[ L(\mathbf{t},D)-L(\mathbf{s},D)\right] ^{n}\right\}
\leq K_{2,3}^{n}(n!)^{2-\eta /2}\left\Vert \mathbf{t}-\mathbf{s}%
\right\Vert^{n\gamma }\nu (D)^{(n-1)\eta /2+1},
\end{equation*}
with $\eta =\frac{\beta }{2}-\left( \alpha -2\right) \gamma >0.$
\end{lemma}

\begin{lemma}
\label{Lem:SumCl} Under condition \textbf{(A)}, there exists a positive
constant $K_{2,4}$ depending on $K_{0}$ and $\alpha $, such that for any $%
\theta \in \left( 0,\delta _{0}\right) $, and positive integers $L<\frac{%
K_{2,4}}{\theta },$ we have 
\begin{equation*}
\sum_{\ell =1}^{L}\frac{2\ell +1}{4\pi }C_{\ell }\left\{ 1-P_{\ell }(\cos
\theta )\right\} \leq K_{2,4}L^{4-\alpha }\theta ^{2}\text{ ,}
\end{equation*}%
and, for any integer $U>1,$ 
\begin{equation*}
\sum_{\ell =U}^{\infty }\frac{2\ell +1}{4\pi }C_{\ell }\left\{ 1-P_{\ell
}(\cos \theta )\right\} \leq K_{2,4}U^{2-\alpha }.
\end{equation*}
\end{lemma}

Note that here $\delta _{0}$ is the constant defined in Lemma \ref{Lem:C1C2}.

In order to obtain the exact Hausdorff measure of the level sets in Theorem %
\ref{Th:Main}, we will also exploit the following two lemmas from Talagrand 
\cite{Talagrand95}. Let $\{f(x),x\in M\}$ be a centered Gaussian field
indexed by $M$ and let $d_{f}(x,y)=\sqrt{\mathbb{E}[(f(x)-f(y))^{2}]}$ be
its canonical metric on $M$. For a $d_{f}$-compact manifold $M$, denote by $%
N_{d_{f}}\left( M,\varepsilon \right) $ the smallest number of balls of
radius $\varepsilon $ in metric $d_{f}$ that are needed to cover $M.$

\begin{lemma}
\label{DudleyLB} If $N_{d_{f}}( M,\varepsilon ) \leq \Psi (\varepsilon ) $
for all $\varepsilon >0$ and the function $\Psi $ satisfies 
\begin{equation*}
\frac{1}{K_{2,5}}\Psi ( \varepsilon ) \leq \Psi ( \frac{\varepsilon }{2})
\leq K_{2,5}\Psi ( \varepsilon ) , \quad \forall \,\varepsilon >0,
\end{equation*}
where $K_{2,5}>0$ is a finite constant. Then 
\begin{equation*}
\mathbb{P}\left\{ \sup_{s,t\in M}\left\vert f( s) -f(t) \right\vert \leq
u\right\} \geq \exp \left( -K_{2,6}\Psi (u) \right) ,
\end{equation*}
where $K_{2,6}>0$ is a constant depending only on $K_{2,5}$.
\end{lemma}

\begin{lemma}
\label{DudleyUB} Let $\{f(x),x\in M\}$ be a centered Gaussian field $a.s.$
bounded on a $d_{f}$-compact set $M$. There exists a universal constant $%
K_{2,7}>0$ such that for any $u>0,$ we have 
\begin{equation*}
\mathbb{P}\left\{ \sup_{x\in M}f(x) \geq K_{2,7}\left( u+\int_{0}^{\overline{%
d}}\sqrt{\log N_{d}( M,\varepsilon ) }d\varepsilon \right)\right\} \leq \exp
\left( -\frac{u^{2}}{\overline{d}^{2}}\right) ,
\end{equation*}
where $\overline{d}=\sup \left\{ d_{f}( x,y) :x,y\in M\right\}$ is the
diameter of $M$ in the metric $d_{f}$.
\end{lemma}

Based on Lemmas \ref{Lem:C1C2} and \ref{DudleyUB} above, we obtain the
following result:

\begin{lemma}
\label{LemTupper} Under the condition (\textbf{A}) with $2<\alpha <4$, there
exist positive constants $K_{2,8}$ and $K_{2,9}$ depending only on $\alpha $
and $K_{0}$, such that for any $z\in \mathbb{S}^{2}$ and $0<r<\delta _{0},$
we have for any $u>K_{2,8}r^{( \alpha -2) /2},$ 
\begin{equation}  \label{Eq:T-tail}
\mathbb{P}\left\{ \sup_{x,y\in D( z,r) }\left\vert T_{0}(x) -T_{0}( y)
\right\vert \geq u\right\} \geq \exp \left( -\frac{u^{2}}{K_{2,9}\left\vert
\rho _{\alpha }( 2r) \right\vert^{2}}\right) .
\end{equation}
\end{lemma}

\begin{proof}
Recall (\ref{variogram}) in Lemma \ref{Lem:C1C2}, we have 
\begin{equation*}
\sqrt{K_{2,1}^{-1}}\rho _{\alpha }( x,y) \leq d_{T_{0}}(x,y) \leq \sqrt{%
K_{2,1}}\rho _{\alpha }( x,y) .
\end{equation*}
It follows immediately that, for any $D( z,r) \subset \mathbb{S}^{2},$ and
any $\epsilon \in ( 0,r) ,$ 
\begin{equation*}
N_{d_{\mathbf{T}}}\left( D( z,r) ,\epsilon \right) \leq \frac{2\pi r^{2}}{%
\pi \left( \epsilon /\sqrt{K_{2,1}}\right) ^{4/( \alpha-2) }} \leq 2(
K_{2,1}) ^{2/( \alpha -2) }\frac{r^{2}}{\epsilon ^{4/( \alpha -2) }},
\end{equation*}
and 
\begin{equation*}
\overline{d}=\sup \left\{ d_{T_{0}}( x,y) :x,y\in D(z,r) \right\} \leq \sqrt{%
K_{2,1}}\rho _{\alpha }( 2r) ,
\end{equation*}
whence 
\begin{equation*}
\begin{split}
& \int_{0}^{\overline{d}}\sqrt{\log N_{d_{\mathbf{T}}}\left( D(z,r)
,\epsilon \right) }d\epsilon \\
\leq & \int_{0}^{\sqrt{K_{2,1}}\rho _{\alpha }( 2r) }\sqrt{\log\left\{
\left( 2( dK_{2,1}) ^{2/( \alpha -2) }\frac{r^{2}}{\epsilon ^{4/( \alpha -2)
}}\right) \vee 2\right\} }d\epsilon \\
\leq & 2\sqrt{\frac{K_{2,1}}{\alpha -2}}\,r^{\alpha
/2-1}\int_{1}^{+\infty}ud(-e^{-u^{2}}) \leq C_{2,1}\,r^{\alpha /2-1},
\end{split}%
\end{equation*}
where $( \cdot \vee \cdot ) $ denotes as usual the maximum function. 
Taking $K_{2,7}=2K_{2,6}C_{2,1}$, then by exploiting Lemma \ref{DudleyUB},
we derive \eqref{Eq:T-tail}. This proves Lemma \ref{LemTupper}.
\end{proof}

Finally, we recall briefly the definition of Hausdorff measure on the sphere 
$\mathbb{S}^{2}$ and an upper density theorem due to \cite{RogTaylor}, which
will be used for proving the lower bound in (\ref{Th:Measure}) Theorem \ref%
{Th:Main}. We refer to Falconer \cite{Fal90} for more information on
geometry of fractals, and to \cite{Xiao09} for its applications in studying
sample path properties of Gaussian random fields.

For some $\varepsilon >0,$ let $\Phi $ be the class of functions $\phi_{1}:(
0,\varepsilon) \rightarrow ( 0,1) ,$ which are right-continuous,
monotonically increasing, with $\phi _{1}(0^{+}) =0$ and such that there
exists a finite constant $K>0$ for which 
\begin{equation*}
\frac{\phi _{1}( 2r) }{\phi _{1}( r) }\leq K,\text{ for }0<r<\frac{1}{2}%
\varepsilon .
\end{equation*}
The $\phi _{1}$-Hausdorff measure of $E\subseteq \mathbb{S}^{2}$ is then
defined as usual by 
\begin{equation*}
\phi _{1}\text{-}m( E) =\liminf_{\varepsilon \rightarrow 0}\left\{
\sum_{i=1}^{\infty }\phi _{1}( 2r_{i}) :E\subseteq\cup _{i=1}^{\infty }D(
x_{i},r_{i}) ,r_{i}<\varepsilon \right\} .
\end{equation*}
Recall that $D\left( x,r\right) $ denotes the open disk of radius $r$
centered at $x\in \mathbb{S}^{2}$. Likewise, the Hausdorff dimension of $E$
is defined by 
\begin{equation*}
\dim E=\inf \left\{ \beta >0:s^{\beta }\text{-}m( E) =0\right\} =\sup
\left\{ \beta >0:s^{\beta }\text{-}m( E) =\infty \right\} .
\end{equation*}

The following lemma is derived from the results in \cite{RogTaylor}, and it
allows to obtain lower bounds for $\phi _{1}$-$m\left( E\right) .$ Recall
first that for any Borel measure $\mu $ on $\mathbb{S}^{2}$ and $\phi_{1}\in
\Phi ,$ the upper $\phi _{1}$-density of $\mu $ at $x\in \mathbb{S}^{2}$ is
defined by 
\begin{equation*}
\overline{D}_{\mu }^{\phi }( x) :=\limsup_{r\rightarrow 0}\frac{\mu \left(
D( x,r) \right) }{\phi _{1}(2r) }.
\end{equation*}

\begin{lemma}
\label{Upper density} For a given $\phi _{1}\in \Phi ,$ there exists a
positive constant $K_{2,10}$ such that for any Borel measure $\mu $ on $%
\mathbb{S}^{2}$ and every Borel set $E\subseteq \mathbb{S}^{2},$ we have 
\begin{equation*}
\phi _{1}\text{-}m( E) \geq K_{2,10}\,\mu ( E)\inf_{x\in E}\left\{ \overline{%
D}_{\mu }^{\phi _{1}}( x) \right\}^{-1}.
\end{equation*}
\end{lemma}

\bigskip

\section{Existence of Nonempty Level Sets and Capacity \label{Sec:Capacity}}

In this section we will provide two more equivalent conditions for the
existence of nonempty level sets. As an immediate consequence, we give a
proof for (\ref{Th:Existence}) in Theorem \ref{Th:Main}.

Let us first introduce the capacity of the random field $\mathbf{T.}$ For
any $\mathbf{t}_{2}\in 
\mathbb{R}
^{2},$ the joint density of $T_{0}(x)$ and $T_{0}(y)$ can be represented as
follows: 
\begin{eqnarray}
&&p\left( d_{\mathbb{S}^{2}}(x,y);\mathbf{t}_{2}\right)  \notag \\
&=&:\frac{1}{2\pi \sigma \left( d_{\mathbb{S}^{2}}(x,y)\right) \sqrt{%
1-\sigma ^{2}(d_{\mathbb{S}^{2}}(x,y))/4}}\exp \left\{ -\frac{1}{2}\mathbf{t}%
_{2}^{T}\Sigma ^{-1}\mathbf{t}_{2}\right\} ,  \label{def:density-p}
\end{eqnarray}%
where $\Sigma ^{-1}$ is the inverse of positive definite covariance matrix $%
\Sigma (d_{\mathbb{S}^{2}}(x,y))$ of $T_{0}\left( x\right) $ and $T_{0}(y)$.
Hence, the joint density of $\mathbf{T}(x)$ and $\mathbf{T}(y)$ at $(\mathbf{%
a},\mathbf{a})$ for any $\mathbf{a}=(a_{1},...,a_{d})^{T}\in \mathbb{R}^{d}$
can be represented as 
\begin{equation}
\Phi \left( d_{\mathbb{S}^{2}}(x,y);\mathbf{a}\right)
=\prod\limits_{j=1}^{d}p\left( d_{\mathbb{S}^{2}}(x,y);a_{j}\mathbf{1}%
\right) .  \label{def:Phi_a}
\end{equation}%
due to the independence of $T_{1},...,T_{d}.$ Obviously, for any $\mathbf{a}%
\in \mathbb{R}^{d},$ 
\begin{equation}
\Phi \left( d_{\mathbb{S}^{2}}(x,y);\mathbf{a}\right) \leq \Phi \left( d_{%
\mathbb{S}^{2}}(x,y)\right) .  \label{ineq:UPhi_a}
\end{equation}%
where $\Phi (\cdot )=\Phi (\cdot ;0),$ Now the $\Phi $-$capacity$ of a Borel
set $E\subseteq \mathbb{S}^{2}$ is then defined as 
\begin{equation*}
\mathcal{C}_{\Phi }(E)=\left\{ \inf_{\mu \in \mathcal{P}(E)}\int \int \Phi
\left( d_{\mathbb{S}^{2}}(x,y)\right) \mu (dx)\mu (dy)\right\} ^{-1},
\end{equation*}%
where $\mathcal{P}(E)$ is the collection of all probability measures on $E$.
Moreover, for any $\mu \in \mathcal{P}\left( E\right) ,$ we define the $\Phi 
$-$energy$ of $\mu $ by 
\begin{equation}
{\large \varepsilon }_{\Phi }(\mu )=\int \int_{\mathbb{S}^{2}\times \mathbb{S%
}^{2}}\Phi (d_{\mathbb{S}^{2}}(x,y))\mu (dx)\mu (dy).  \label{def:Phi-energy}
\end{equation}%
It is readily seen that 
\begin{equation*}
\mathcal{C}_{\Phi }(E)=\left\{ \inf_{\mu \in \mathcal{P}(E)}{\large %
\varepsilon }_{\Phi }(\mu )\right\} ^{-1}.
\end{equation*}%
See \cite{DalangMuellerXiao,KhoshXiao} for more informations about $\Phi $-$%
capacity$ and $\Phi $-$energy$ for probability measures on Borel sets in
Euclidean spaces.

Let $N$ be the north pole of the unit sphere, then we have the following
result.

\begin{proposition}
\label{Prop:Capacity} Under the conditions of Theorem \ref{Th:Main}, the
following are equivalent:

\begin{enumerate}
\item[(i)] $\mathcal{C}_{\Phi }( D( N,r) ) >0$ for some $r>0;$

\item[(ii)] for all $\mathbf{a\in \mathbb{R}}^{d}$ and $x\in \mathbb{S}^{2},$
$\mathbb{P}\left\{ \mathbf{T}^{-1}(\mathbf{a}) \cap D( x,r) \neq \varnothing
\right\} >0$ for some $r>0;$

\item[(iii)] $\Phi ( d_{\mathbb{S}^{2}}( N,\cdot ) )\in L^{1}( D( N,r) ) ,$
for some $r>0.$
\end{enumerate}
\end{proposition}

\begin{proof}
We employ an argument that is similar, in spirit, to that used by
Khoshnevisan and Xiao \cite{KhoshXiao} in the proof of its Theorem 2.9. The
steps of our proof is taken in the order $(i)\implies (ii)\implies
(iii)\implies (i).$

To prove $(i)\implies (ii),$ it is sufficient to prove 
\begin{equation*}
\mathbb{P}\left\{ \mathbf{a\in T}(D(x,r))\right\} >0.
\end{equation*}%
The idea is to find a random measure $\hbar $ such that 
\begin{equation*}
\mathbb{P}\left\{ {\large \hbar }\left( \mathbf{T}^{-1}(\mathbf{a})\cap
D(x,r)\right) >0\right\} >0.
\end{equation*}%
More precisely, it satisfies 
\begin{equation}
\mathbb{P}\left\{ {\large \hbar }(D(x,r))>0\right\} >0,  \label{PrHsphere>0}
\end{equation}%
as well as ${\large \hbar }(D_{T}(\delta ))=0,\ a.s.$ for any $\delta >0$
with 
\begin{equation}
D_{T}(\delta )=\left\{ y\in D(x,r):\left\vert \mathbf{T}(y)-\mathbf{a}%
\right\vert _{\infty }>\delta \right\} .  \label{def:DT}
\end{equation}%
Note that, under condition $(i),$ there exists a probability measure $\mu
\in \mathcal{P}(\mathbb{S}^{2})$ such that the $\Phi $-$energy$ of $\mu $ in 
$\left( \ref{def:Phi-energy}\right) $ is finite. Now for any $\varepsilon
>0,\ \mathbf{a}\in \mathbb{R}^{d},$ we define a random measure 
\begin{equation}
{\large \hbar }_{\varepsilon }(E)=(2\varepsilon )^{-d}\int_{E}\mathrm{1}%
\left\{ \left\vert \mathbf{T}(y)-\mathbf{a}\right\vert _{\infty }\leq
\varepsilon \right\} \mu (dy).  \label{def:H}
\end{equation}%
where $E$ is any Borel set on $\mathbb{S}^{2}.$ Then it is readily seen that 
\begin{eqnarray*}
(2\varepsilon )^{d}\mathbb{E}\left[ {\large \hbar }_{\varepsilon }(E)\right]
&=&\int_{E}\mathbb{P}\left\{ \left\vert \mathbf{T}(y)-\mathbf{a}\right\vert
_{\infty }\leq \varepsilon \right\} \mu (dy) \\
&=&\int_{E}\left[ \prod\limits_{j=1}^{d}\int_{-\varepsilon }^{\varepsilon }%
\frac{1}{\sqrt{2\pi }}\exp \left\{ -\frac{1}{2}\left( t_{j}+a_{j}\right)
^{2}\right\} dt_{j}\right] \mu (dx).
\end{eqnarray*}%
By Fatou's Lemma, we obtain 
\begin{equation}
\lim_{\varepsilon \rightarrow 0}\mathbb{E}\left[ {\large \hbar }%
_{\varepsilon }(E)\right] =(2\pi )^{-d/2}\exp \left\{ -\frac{1}{2}\left\Vert 
\mathbf{a}\right\Vert ^{2}\right\} \mu (E),  \label{limEH}
\end{equation}%
where $\left\Vert \cdot \right\Vert $ denotes the Euclidean distance in $%
\mathbb{R}^{d}.$ Moreover, 
\begin{eqnarray*}
&&(2\varepsilon )^{2d}\mathbb{E}\left\{ \left[ {\large \hbar }_{\varepsilon
}(E)\right] ^{2}\right\} \\
&=&\int_{E}\int_{E}\mathbb{P}\left\{ \left\vert \mathbf{T}(y)-\mathbf{a}%
\right\vert _{\infty }\leq \varepsilon ,\left\vert \mathbf{T}(z)-\mathbf{a}%
\right\vert _{\infty }\leq \varepsilon \right\} \mu (dy)\mu (dz) \\
&=&\int_{E}\int_{E}\left[ \int_{-\varepsilon }^{\varepsilon
}\int_{-\varepsilon }^{\varepsilon }p\left( d_{\mathbb{S}^{2}}(y,z);\mathbf{t%
}_{2}+a_{j}\mathbf{1}\right) d\mathbf{t}_{2}\right] ^{d}\mu (dy)\mu (dz).
\end{eqnarray*}%
Recall $(\ref{def:density-p}),$ we have for any $\mathbf{t}_{2}\in \left[
-\varepsilon ,\varepsilon \right] \times \left[ -\varepsilon ,\varepsilon %
\right] ,$ 
\begin{equation}
\lim_{\varepsilon \rightarrow 0}p\left( d_{\mathbb{S}^{2}}(y,z);\mathbf{t}%
_{2}+a_{j}\mathbf{1}\right) =p\left( d_{\mathbb{S}^{2}}(y,z);a_{j}\mathbf{1}%
\right) .  \label{ineq:p(t)}
\end{equation}%
Thus, by Fatou's Lemma again and the fact of $(i),$ we have 
\begin{eqnarray}
\lim_{\varepsilon \rightarrow 0}\mathbb{E}\left\{ \left[ {\large \hbar }%
_{\varepsilon }\left( E\right) \right] ^{2}\right\} &=&\int_{E}\int_{E}\Phi
\left( d_{\mathbb{S}^{2}}(y,z);\mathbf{a}\right) \mu (dy)\mu (dz)  \notag \\
&\leq &\int_{E}\int_{E}\Phi \left( d_{\mathbb{S}^{2}}(y,z)\right) \mu
(dy)\mu (dz)<\infty .  \label{LimEH^2}
\end{eqnarray}%
in view of the inequality (\ref{ineq:UPhi_a}). Now let $E=D(x,r),$ it is
readily seen that for ${\large \hbar =}\lim_{\varepsilon \rightarrow 0}%
{\large \hbar }_{\varepsilon }$ in the weakly convergent sense, we have 
\begin{equation*}
\mathbb{E}\left[ {\large \hbar }(D(x,r))\right] =(2\pi )^{-d/2}\mu
(D(x,r)),\ \mathbb{E}\left\{ \left[ {\large \hbar }(D(x,r))\right]
^{2}\right\} <\infty ,
\end{equation*}%
in view of $(\ref{limEH})$\ and $(\ref{LimEH^2}).$ Hence, by Cauchy-Schwarz
inequality, we obtain 
\begin{equation*}
\mathbb{P}\left\{ {\large \hbar }(D(x,r))>0\right\} \geq \left\vert \mathbb{E%
}\left[ {\large \hbar }(D(x,r))\right] \right\vert ^{2}\left\{ \mathbb{E}%
\left[ \left\vert {\large \hbar }(D(x,r))\right\vert ^{2}\right] \right\}
^{-1}>0.
\end{equation*}%
It remains to prove that the random measure ${\large \hbar }$ is supported
on $\mathbf{T}^{-1}\left( \mathbf{a}\right) ,$ which can be obtained by the
fact that for each $\delta >0,$ 
\begin{equation*}
\mathbb{E}\left[ {\large \hbar }\left( D_{T}(\delta )\right) \right]
=\lim_{\varepsilon \rightarrow 0}\mathbb{E}\left[ {\large \hbar }%
_{\varepsilon }\left( D_{T}(\delta )\right) \right] =0,
\end{equation*}%
in view of the definitions $D_{T}(\delta )$ and ${\large \hbar }%
_{\varepsilon }$ in (\ref{def:DT}) and (\ref{def:H}).

The implication of $(iii)\implies (i)$ is obvious and we leave the proof of $%
(ii)\implies (iii)$ in the following Lemma. Thus the proof is completed.
\end{proof}

\begin{lemma}
In Proposition \ref{Prop:Capacity}, $(ii)\implies (iii).$
\end{lemma}

\begin{proof}
Recall (\ref{variogram-sig}) and note that for any $x,y\in D(N,r),$ 
\begin{equation*}
\mathbb{E}\left( T_{0}(x)|T_{0}(y)\right) =\frac{\mathbb{E}\left(
T_{0}(x)T_{0}(y)\right) }{\mathbb{E}\left( T_{0}(y)^{2}\right) }%
T_{0}(y)=C(x,y)T_{0}(y),
\end{equation*}%
where 
\begin{equation*}
C(x,y)=:\left( 1-\frac{1}{2}\sigma ^{2}(d_{\mathbb{S}^{2}}(x,y))\right) .
\end{equation*}%
Obviously, we have 
\begin{equation*}
\left( 1-\frac{1}{2}\sigma ^{2}(2r)\right) \leq C(x,y)\leq 1.
\end{equation*}%
Now let $Q_{0}(x,y)=T_{0}(x)-C(x,y)T_{0}(y),$ then it is readily seen that $%
Q_{0}(x,y)$ is again Gaussian with $\mathbb{E}\left[ Q_{0}(x,y)\right] =0,$
and 
\begin{equation}
\mathrm{Var}\left( Q_{0}(x,y)\right) =1-C(x,y)^{2}=\left[ 2\pi p\left( d_{%
\mathbb{S}^{2}}(x,y);\mathbf{0}\right) \right] ^{-2}  \label{VarZ}
\end{equation}%
in view of (\ref{def:density-p}). Moreover, we have the following
decomposition for $\mathbf{T}(x)$: 
\begin{equation*}
\mathbf{T}(x)=\mathbf{Q}(x,y)+C(x,y)\mathbf{T}(y),
\end{equation*}%
with $\mathbf{Q}(x,y)$ and $\mathbf{T}(y)$ independent. Now we are ready to
prove $(ii)\implies (iii).$ For any $\varepsilon \in \left( 0,1\right) ,$
define a random measure 
\begin{equation*}
{\large \hbar }_{\varepsilon }^{0}(B)=(2\varepsilon )^{-d}\int_{B}\mathrm{1}%
\left\{ \left\vert \mathbf{T}(x)\right\vert _{\infty }\leq \varepsilon
\right\} d\nu (x).
\end{equation*}%
for any Borel set $B\subseteq \mathbb{S}^{2},$ and consider the following
conditional expectation 
\begin{equation*}
M(y,\varepsilon )=\mathbb{E}\left[ {\large \hbar }_{\varepsilon
}^{0}(D(x,r))|\mathbf{T}(y)\right] ,\quad y\in D(x,r).
\end{equation*}%
Then by the independence of $\mathbf{Q}(x,y)$ and $\mathbf{T}(y),$ we have 
\begin{equation*}
M(y,\varepsilon )\geq \left( 2\varepsilon \right) ^{-d}\int_{D(N,r)}\mathbb{P%
}\left\{ \left\vert \mathbf{Q}(x,y)\right\vert _{\infty }\leq \frac{%
\varepsilon }{2}\right\} d\nu (x)\mathrm{1}\left\{ \left\vert \mathbf{T}%
(y)\right\vert _{\infty }\leq \frac{\varepsilon }{2}\right\} .
\end{equation*}%
Now, by first squaring both sides of the inequality above, then doing the
expectation after taking the supremum over $D(N,r),$ we obtain 
\begin{eqnarray}
&&\mathbb{E}\left\{ \sup_{y\in D(N,r)}M^{2}(y,\varepsilon )\right\}   \notag
\\
&\geq &(2\varepsilon )^{-2d}\inf_{y\in D(N,r)}\left[ \int_{D(N,r)}\mathbb{P}%
\left\{ \left\vert \mathbf{Q}(x,y)\right\vert _{\infty }\leq \frac{%
\varepsilon }{2}\right\} d\nu (x)\right] ^{2}  \label{ineq:EsupM2-0} \\
&&\cdot \mathbb{P}\left( \left\vert \mathbf{T}(y)\right\vert \leq \frac{%
\varepsilon }{2}\text{ for some }y\in D(N,r)\right) .  \notag
\end{eqnarray}%
Note that, by Jensen's inequality, we have 
\begin{equation}
\mathbb{E}\left[ \sup_{y\in D(N,r)}M^{2}(y,\varepsilon )\right] \leq \mathbb{%
E}\left\{ \sup_{y\in D(N,r)}\mathbb{E}\left[ \left\vert {\large \hbar }%
_{\varepsilon }^{0}\left( D(x,r)\right) \right\vert ^{2}|\mathbf{T}(y)\right]
\right\} .  \label{ineq:EsupM2}
\end{equation}%
Moreover, recall Lemma \ref{Lem:C1C2} and $\mathrm{Var}\left(
T_{0}(x)\right) =1$ for any $x\in \mathbb{S}^{2},$ we have for any three
points $x_{1},x_{2},x_{3}\in D(N,r)$ with some $r\in (0,\delta _{0}),$ 
\begin{eqnarray*}
&&\det \left[ \mathrm{Cov}\left(
T_{0}(x_{1})T_{0}(x_{2})|T_{0}(x_{3})\right) \right]  \\
&=&\mathrm{Var}(T_{0}(x_{2}))\mathrm{Cov}\left(
T_{0}(x_{1})|T_{0}(x_{2}),T_{0}(x_{3})\right)  \\
&\geq &C_{3,1}\left[ p\left( \min_{i=2,3}d_{\mathbb{S}^{2}}(x_{1},x_{i});%
\mathbf{0}\right) \right] ^{-2}.
\end{eqnarray*}%
where the constant $C_{3,1}$ is positive and depends on $\delta _{0},$ $%
K_{2,1}$ and $K_{2,2}.$ Hence, 
\begin{eqnarray*}
(\ref{ineq:EsupM2}) &\leq &\mathbb{E\{}\sup_{y\in D(N,r)}(2\varepsilon
)^{-2d}\int_{D(N,r)}\int_{D(N,r)}d\nu (x_{1})d\nu (x_{2}) \\
&&\cdot \mathbb{P}\left\{ \left\vert \mathbf{T}(x_{1})\right\vert _{\infty
}\leq \varepsilon ,\left\vert \mathbf{T}(x_{2})\right\vert _{\infty }\leq
\varepsilon |\mathbf{T}(y)\right\} \} \\
&\leq &\mathbb{E}\sup_{x_{3}\in D(N,r)}\int_{D(N,r)}\int_{D(N,r)}d\nu
(x_{1})d\nu (x_{2}) \\
&&\cdot (2\pi )^{-d}\left\{ \det \left[ Cov\left(
T_{0}(x_{1}),T_{0}(x_{2})|T_{0}(x_{3})\right) \right] \right\} ^{-d/2} \\
&\leq &C_{3,2}\sup_{x_{3}\in D(N,r)}\int_{D(N,r)}\int_{D(N,r)}\left[ p\left(
\min_{i=2,3}d_{\mathbb{S}^{2}}(x_{1},x_{i});\mathbf{0}\right) \right]
^{-2}d\nu (x_{1})d\nu (x_{2}) \\
&\leq &C_{3,2}\int_{D(N,r)}\int_{D(N,r)}\Phi (d_{\mathbb{S}%
^{2}}(x_{1},x_{2}))d\nu (x_{1})d\nu (x_{2}) \\
&&+C_{3,2}\ \nu \left\{ D(N,r)\right\} \sup_{x_{3}\in
D(N,r)}\int_{D(N,r)}\Phi \left( d_{\mathbb{S}^{2}}(x,x_{3})\right) d\nu (x),
\end{eqnarray*}%
which gives 
\begin{equation}
\mathbb{E}\left[ \sup_{y\in D(N,r)}M^{2}(y,\varepsilon )\right] \leq
2C_{3,2}\ \nu \left\{ D(N,r)\right\} \int_{D(N,r)}\Phi \left( d_{\mathbb{S}%
^{2}}(x,N)\right) d\nu (x).  \label{ineq:EsupM2-1}
\end{equation}%
due to the fact that, for any $y\in D(N,r)$ 
\begin{equation*}
\int_{D(N,r)}\Phi \left( d_{\mathbb{S}^{2}}(x,y)\right) d\nu (x)\leq
\int_{D(N,r)}\Phi \left( d_{\mathbb{S}^{2}}(x,N)\right) d\nu (x)
\end{equation*}%
Here the constant $C_{3,2}$ is positive and depends on $C_{3,1}.$ In the
meantime, by Fatou's Lemma and the upper bound in $(\ref{ineq:p(t)})$, we
have 
\begin{eqnarray}
&&\lim_{\varepsilon \rightarrow 0}(2\varepsilon )^{-d}\left[ \int_{D(N,r)}%
\mathbb{P}\left\{ \left\vert \mathbf{Z}(x,y)\right\vert _{\infty }\leq \frac{%
\varepsilon }{2}\right\} d\nu (x)\right]   \notag \\
&=&(2\pi )^{d/2}\int_{D(N,r)}\Phi \left( d_{\mathbb{S}^{2}}(x,y)\right) d\nu
(x)  \notag \\
&\geq &\frac{1}{3}(2\pi )^{d/2}\int_{D(N,r)}\Phi \left( d_{\mathbb{S}%
^{2}}(x,N)\right) d\nu (x)  \label{limPZ}
\end{eqnarray}%
in view of the representation $(\ref{VarZ})$ and the fact that, for any $%
y\in D(N,r)$ 
\begin{equation*}
\int_{D(N,r)}\Phi \left( d_{\mathbb{S}^{2}}(x,y)\right) d\nu (x)\geq \frac{1%
}{3}\int_{D(N,r)}\Phi \left( d_{\mathbb{S}^{2}}(x,N)\right) d\nu (x)
\end{equation*}%
Finally, combining inequalities (\ref{ineq:EsupM2-0}), (\ref{ineq:EsupM2-1})
and (\ref{limPZ}), together with 
\begin{eqnarray*}
&&\mathbb{P}\left( \left\vert \mathbf{T}(y)\right\vert \leq \frac{%
\varepsilon }{2}\text{ for some }y\in D(N,r)\right)  \\
&\geq &\mathbb{P}\left( \mathbf{T}(y)=\mathbf{0}\text{ for some }y\in
D(N,r)\right) >0,
\end{eqnarray*}%
we obtain 
\begin{equation*}
C_{3,3}\left[ \int_{D(N,r)}\Phi \left( d_{\mathbb{S}^{2}}(x,N)\right) d\nu
(x)\right] ^{-1}\geq \mathbb{P}\left( \mathbf{T}^{-1}(\mathbf{0})\cap
D(N,r)\neq \varnothing \right) >0,
\end{equation*}%
where the constant $C_{3,3}$ is positive and depends on $C_{3,1}$ and $%
C_{3,2}.$ and the proof of Lemma is then completed.
\end{proof}

As an immediate consequence, we prove the follows:

\begin{proof}[Proof of (\protect\ref{Th:Existence}) in Theorem \protect\ref%
{Th:Main}]
Recalling the formula (\ref{variogram-sig}) and (\ref{def:Phi_a}) for the
definition $\Phi =\Phi _{\mathbf{0}}$ and the estimations (\ref{variogram})
for variogram of $T_{0}$ in Lemma \ref{Lem:C1C2}, we have 
\begin{equation*}
\int_{D(N,r)}\Phi \left( d_{\mathbb{S}^{2}}(x,N)\right) d\nu (x)\approx
\int_{0}^{r}\int_{0}^{2\pi }\theta ^{d(1-\alpha /2)}\sin \theta d\phi
d\theta ,
\end{equation*}%
where by $A\approx B$ we mean that $C_{3,4}B\leq A\leq C_{3,5}B$ for some
positive constants $C_{3,4},$ $C_{3,5}$ depending on $K_{2,1}.$ The right
side of the equivalence above is finite if and only $1+d(1-\alpha /2)>-1,$
thus the equivalence (\ref{Th:Existence}) is proved in view of Proposition %
\ref{Prop:Capacity}.
\end{proof}

\section{Hausdorff Measure of the Level Sets \label{Sec:Measure}}

In this section, we give a proof of (\ref{Th:Measure}) in Theorem \ref%
{Th:Main}, which is divided into proving lower and upper bounds separately.
For the lower bound we will make the local time $L({\mathbf{t}},\cdot )$ as
a natural measure on $\mathbf{T}^{-1}\left( \mathbf{t}\right) $ and make use
of Lemma \ref{Upper density}.

The proof of upper bound is more involved. We will extend the covering
argument by Baraka and Mountford \cite{BarakMount} to the spherical random
fields. Their method strengthened the covering argument by Xiao \cite{Xiao97}
for the level sets, see also Talagrand \cite{Talagrand95,Talagrand98}, .

\subsection{Lower bound for the Hausdorff measure}

We first establish the lower bound for the exact Hausdorff measure of the
level set $\mathbf{T}^{-1}\left( \mathbf{t}\right) $ by applying Lemma \ref%
{Upper density}.

A basic tool to establish it is the proposition below:

\begin{proposition}
\label{Prop:LIL} Under the conditions of Theorem \ref{Th:Main}, there exists
a positive constant $K_{4,1},$ such that for all $x\in \mathbb{S}^{2}$ and $%
\mathbf{t}\in \mathbb{R}^{d}$, with probability one, 
\begin{equation*}
\limsup_{r\rightarrow 0}\frac{L\left( \mathbf{t},D(x,r) \right) }{\phi ( r) }%
\leq K_{4,1},
\end{equation*}
where $\phi ( \cdot ) $ is defined in (\ref{def:psi}).
\end{proposition}

\begin{proof}
Denote by $f_{m}(x)=L\left( \mathbf{t},D(x,2^{-m})\right) ,$ the local time
of $\mathbf{T}$ at $\mathbf{t\in \mathbb{R}}^{d}$ in $D(x,2^{-m})\subseteq 
\mathbb{S}^{2},$ for every integer $m\geq 1,$ and recall the representation (%
\ref{repLT}), we have 
\begin{equation*}
\mathbb{E}\left[ \int_{\mathbb{S}^{2}}\left[ f_{m}(x)\right] ^{n}L\left( 
\mathbf{t},d\nu (x)\right) \right] =(2\pi )^{-(n+1)}
\end{equation*}%
\begin{equation*}
\times \int_{\mathbb{S}^{2}}\int_{\left[ D(x,2^{-m})\right] ^{n}}\int_{%
\mathbb{R}^{d(n+1)}}\exp \left\{ i\sum_{j=1}^{n+1}\boldsymbol{\xi }%
_{j}^{T}\left( \mathbf{T}(x_{j})-\mathbf{t}\right) \right\} d\boldsymbol{\xi 
}d\boldsymbol{\nu }(\mathbf{x}),
\end{equation*}%
with $d\boldsymbol{\nu }(\mathbf{x})=d\nu (x_{1})\cdots d\nu (x_{n+1})$ and $%
d\boldsymbol{\xi =}d\boldsymbol{\xi }_{1}\cdots d\boldsymbol{\xi }_{n+1},$
where $\boldsymbol{\xi }_{1},\cdots ,\mathbb{R}^{d}.$ Using the similar
argument as in \cite{LanXiao}, we obtain that the right-hand side of the
equality above is bounded by 
\begin{equation}
\begin{split}
& \int_{\mathbb{S}^{2}}\int_{\left[ D(x,2^{-m})\right] ^{n}}\frac{d\nu (%
\mathbf{x})}{\left\vert \det Cov\left( T(x_{1}),...,T(x_{n+1})\right)
\right\vert ^{1/2}} \\
& \leq (C_{4,1})^{n+1}(n!)^{\frac{d}{4}(\alpha -2)}\left( 2^{-m}\right) ^{%
\frac{n}{2}(4-(\alpha -2)d)}.
\end{split}
\label{ineq:EInt-fm^n}
\end{equation}%
where $C_{4,1}$ is a positive constant and depends on $\alpha ,\ d$ and $%
K_{2,2}.$ Now, for each $m\geq 1,$ consider the random set 
\begin{equation*}
D_{m}(\omega )=\left\{ x\in \mathbb{S}^{2}:f_{m}(x)\geq K_{4,1}\phi
(2^{-m})\right\} ,
\end{equation*}%
Let $K_{4,1}>C_{4,1}e^{2}$ and taking $n=\left[ \log m\right] ,$ then by
applying $\left( \ref{ineq:EInt-fm^n}\right) $ and Stirling's formula, we
have 
\begin{eqnarray*}
\mathbb{E}\left[ L\left( \mathbf{t},D_{m}\right) \right] &=&\mathbb{E}\left[
\int_{D_{m}}L\left( \mathbf{t},d\nu (x)\right) \right] \leq \frac{\mathbb{E}%
\int_{\mathbb{S}^{2}}\left[ f_{m}(x)\right] ^{n}L\left( t,d\nu (x)\right) }{%
\left[ K_{4,1}\phi (2^{-m})\right] ^{n}} \\
&\leq &(C_{4,1})^{n+1}(n!)^{d(\alpha -2)/4}\left[ K_{4,1}(\log m)^{d(\alpha
-2)/4}\right] ^{-n}\leq m^{-2},
\end{eqnarray*}%
which implies that 
\begin{equation*}
\mathbb{E}\left[ \sum_{m=1}^{\infty }L(\mathbf{t},D_{m})\right] <\infty .
\end{equation*}%
Therefore, by Borel-Cantelli Lemma, we have for almost all $x\in \mathbb{S}%
^{2},$ with probability 1, 
\begin{equation}
\limsup_{m\rightarrow \infty }\frac{L\left( \mathbf{t},D(x,2^{-m})\right) }{%
\phi (2^{-m})}\leq K_{4,1}.  \label{limsupL-upper}
\end{equation}%
Finally, for any $r>0$ small enough, there always exists an integer $m$ such
that $2^{-m}<r<2^{-m+1}$ and $(\ref{limsupL-upper})$ is applicable. Since
the function $\phi (r)$ is increasing near $r=0,$ the result in Proposition %
\ref{Prop:LIL}\ follows from $(\ref{limsupL-upper})$ and a monotonicity
argument.
\end{proof}

\begin{theorem}
\label{Th:HM-Lower} Under the conditions of Theorem \ref{Th:Main}, there
exists a positive constant $K_{4,2},$ such that for every $\mathbf{t}\in 
\mathbb{R}^{d},$ with probability one, 
\begin{equation*}
\phi \text{-}m\left( \mathbf{T}^{-1}( \mathbf{t}) \right) \geq
K_{4,2}L\left( \mathbf{t},\mathbb{S}^{2}\right) .
\end{equation*}
\end{theorem}

\begin{proof}
Let 
\begin{equation*}
D_{0}=\left\{ x\in \mathbb{S}^{2}:\ \limsup_{r\rightarrow 0}\frac{L\left( 
\mathbf{t},D( x,r) \right) }{\phi ( r) }>K_{4,1}\right\} \text{ .}
\end{equation*}
it is readily seen that $D_{0}$ is a Borel set and $L\left( \mathbf{t}%
,D_{0}\right) =0$ almost surely, in view of Proposition \ref{Prop:LIL}.
Therefore, we have almost surely 
\begin{equation*}
\phi \text{-}m\left( \mathbf{T}^{-1}( \mathbf{t}) \right) \geq \phi \text{-}%
m\left( \mathbf{T}^{-1}( \mathbf{t}) \cap (\mathbb{S}^{2}\backslash D_{0})
\right) \geq \frac{K_{2,10}}{K_{4,1}}L\left( \mathbf{t},\mathbb{S}%
^{2}\backslash D_{0}\right) ,
\end{equation*}
in view of Lemma \ref{Upper density}. Let $K_{4,2}=\frac{K_{2,10}}{K_{4,1}}$%
, then the result in Theorem \ref{Th:HM-Lower} is obtained by the fact that $%
L\left( t,\mathbb{S}^{2}\right) =L\left( \mathbf{t},\mathbb{S}^{2}\backslash
D_{0}\right) $ almost surely.
\end{proof}

\subsection{Upper Bound for the Hausdorff Measure}

Now we start working toward the upper bound for the exact Hausdorff measure
of the level set $\mathbf{T}^{-1}( \mathbf{t}).$

One important ingredient for establishing the upper bound for the exact
Hausdorff measure of the level sets of $\mathbf{T}$ is the following
Proposition \ref{Smooth}; the statement is similar to Proposition 4.1 of 
\cite{Talagrand95}, but indeed much stronger. For any $x\in \mathbb{S}^{2}$
and any $r_{0}\in \left( 0,\delta _{0}\right) ,$ let us consider the event $%
\Omega \left( x,r_{0},C\right) $ defined by 
\begin{eqnarray}
\exists \ r &\in & ( r_{0}^{2},r_{0}) \ \text{such that }\sup_{d_{\mathbb{S}%
^{2}}( x,y) <r}\left\Vert \mathbf{T}( x) -\mathbf{T}( y) \right\Vert \leq
2Cw( r)  \notag \\
&&\left. \text{and }L\left( \mathbf{T}( x) ,D( x,r)\right) >\frac{\pi }{C}%
\phi ( r) \right\} ,  \label{def:SmoothSet}
\end{eqnarray}
where $w( r) =\rho _{\alpha }\left( r/\sqrt{\log \left\vert \log
r\right\vert }\right) ,$ and $L\left( \mathbf{T}( x) ,D(x,r) \right) $ is
the local time at $\mathbf{T}( x) $ defined in the formula $( \ref{defLT}) $
in section \ref{Sec:Preliminaries}. We have the following result:

\begin{proposition}
\label{Smooth} Under the conditions of Theorem \ref{Th:Main}, there exists a
positive constant $K_{4,3}$ depending on $\alpha ,\ d$ and $K_{0},$ such
that for any $r_{0} \in ( 0,\delta _{0}) $ and $x\in \mathbb{S}^{2},$ we
have 
\begin{equation*}
\mathbb{P}\left\{ \Omega _{r_{0}}\left( x,r_{0},K_{4,3}\right) \right\} >1-%
\frac{1}{\left\vert \log r_{0}\right\vert ^{2}}.
\end{equation*}
\end{proposition}

The proof of Proposition \ref{Smooth}\ is quite involved and will be
presented in Section \ref{Sec:PropSmooth}. We now apply this proposition to
prove the following upper bound.

\begin{theorem}
\label{Th:HM-Upper} Under the conditions of Theorem \ref{Th:Main}, there
exists a constant $K_{4,4}>0$ depending on $\alpha ,\ d$ and $K_{0}$ such
that for every $\mathbf{t}\in \mathbb{R}^{d},$ with probability one, 
\begin{equation*}
\phi \text{-}m( \mathbf{T}^{-1}( \mathbf{t}) ) \leq K_{4,4}L\left( \mathbf{t}%
,\mathbb{S}^{2}\right) .
\end{equation*}
\end{theorem}

\begin{proof}
Recall the event $\Omega (x,r_{0},C)$ defined in $(\ref{def:SmoothSet}),$\
for all integers $p\geq 1,$ let 
\begin{equation*}
R_{p}=\left\{ 
\begin{array}{c}
x\in \mathbb{S}^{2}:\exists r\in \left[ 2^{-2^{p}},2^{-2^{p-1}}\right] ,%
\text{ such that }L\left( \mathbf{T}(x),D(x,r)\right) >\frac{\pi }{K_{4,3}}%
\phi (r) \\ 
\text{and }\sup_{d_{\mathbb{S}^{2}}(x,y)\leq r}\left\Vert \mathbf{T}(x)-%
\mathbf{T}(y)\right\Vert \leq 2K_{4,3}w(r)%
\end{array}%
\right\}
\end{equation*}%
and 
\begin{equation*}
\Omega _{p,1}=\left\{ \omega :\nu (R_{p})\geq \nu (\mathbb{S}^{2})(1-\frac{1%
}{2^{p-1}})\right\} .
\end{equation*}%
In words, $\Omega _{p,1}$ is the event that "a large portion of $\mathbb{S}%
^{2}$ consists of points at which $\mathbf{T}$ has the smallest
oscillation". Recall that $\nu \left( \mathbb{S}^{2}\right) =4\pi ,$ thus
the complement of $\Omega _{p,1}$ is 
\begin{equation*}
\Omega _{p,1}^{c}=\left\{ \nu \left( R_{p}\right) <4\pi \left( 1-\frac{1}{%
2^{p-1}}\right) \right\} =\left\{ \nu \left( R_{p}^{c}\right) >4\pi \cdot 
\frac{1}{2^{p-1}}\right\} .
\end{equation*}%
By Markov's inequality and Fubini's theorem, we have 
\begin{eqnarray*}
\mathbb{P}\left\{ \Omega _{p,1}^{c}\right\} &\leq &\frac{2^{p-1}}{4\pi }%
\mathbb{E}(\nu (R_{p}^{c}))=\frac{2^{p-1}}{4\pi }\int_{\mathbb{S}^{2}}%
\mathbb{P}\left\{ y\in R_{p}^{c}\right\} d\nu (y) \\
&\leq &\frac{2^{p-1}}{4\pi }\int_{\mathbb{S}^{2}}\frac{1}{2^{2(p-1)}}d\nu
(y)\leq \frac{1}{2^{p-1}},
\end{eqnarray*}%
where in the second inequality we have used the Proposition \ref{Smooth}.
Therefore, 
\begin{equation}
\sum_{p=1}^{\infty }\mathbb{P}\left\{ \Omega _{p,1}^{c}\right\} <\infty .
\label{Omega1c}
\end{equation}%
Before going to the next stage, we construct \emph{Voronoi} cells on $%
\mathbb{S}^{2}$ for every integer $k\geq 1$ (see \cite{MPbook} as a
reference). For $k=1,$ let $\Xi _{1}=\left\{ x_{1,1},x_{1,2}\right\} $ where 
$x_{1,1}$ and $x_{1,2}$ are north and south poles respectively, and 
\begin{equation*}
\mathcal{V}_{1}(x_{1,i})=\{y\in \mathbb{S}^{2}:d(y,x_{1,i})\leq
d(y,x_{1,j}),\ j\neq i,j=1,2\}
\end{equation*}%
for $i=1,2.$ For $k\geq 2,$ suppose $\Xi _{k-1}$ is the centers of \emph{%
Voronoi} cells of level $k-1$ which has been chosen. Now choose a set of
points $\Xi _{k,i}=\left\{ x_{k},_{i_{1}}...,x_{k,i_{j}}\right\} \in 
\mathcal{V}_{k-1}\left( x_{k-1,i}\right) $ for each $i=1,...,N_{k-1}$ with $%
N_{k-1}$ the cadinality of $\Xi _{k-1},$ such that for any $x_{k,i_{j}}\neq
x_{k,i_{j^{\prime }}^{\prime }},$ $2^{-k}\leq
d(x_{k,i_{j}},x_{k,i_{j^{\prime }}^{\prime }})\leq 2^{1-k}$. The \emph{%
Voronoi cells} of order $k$ is then defined as follows: 
\begin{equation*}
\mathcal{V}_{k}(x_{k,i_{l}})=\{y\in \mathcal{V}%
_{k-1}(x_{k-1,i}):d(y,x_{k,i_{l}})\leq d(y,x_{k,i_{l^{\prime }}}),\text{ for
any }i_{l^{\prime }}\neq i_{l},\}.
\end{equation*}%
Label the points in $\Xi _{k,1},...,\Xi _{k,N_{k-1}}$ by $%
x_{k,1},....,x_{k,N_{k}}$ with $N_{k}$ the cardinality of the union set $%
\cup _{i}\Xi _{k,i},$ and denote by $\Xi _{k}=\left\{
x_{k,1},....,x_{k,N_{k}}\right\} .$ The procedure of constructing \emph{%
Voronoi cells} of order $k$ and higher can be iterated. Obviously, these
cells are nested for different $k^{\prime }s,$ and non-overlapping with $%
\mathbb{S}^{2}=\cup _{i}\mathcal{V}_{k}(x_{k,i})$ for every $k$. The cell $%
\mathcal{V}_{k}(x_{k,i})$ is called "good" if 
\begin{equation}
\sup_{x,y\in \mathcal{V}_{k}(x_{k,i})}\left\Vert \mathbf{T}(x)-\mathbf{T}%
(y)\right\Vert \leq 2K_{4,3}w(2^{-k})  \label{Ineq-H1,p}
\end{equation}%
as well as 
\begin{equation*}
L\left( \mathbf{T}(x_{k,i}),\mathcal{V}_{k}(x_{k,i})\right) >\frac{\pi }{%
K_{4,3}}\phi (2^{-k}).
\end{equation*}%
By Proposition \ref{Smooth}, we can find a family $\mathcal{H}_{1,p}$ of
good cells for some order $k\in \left[ 2^{p-1},2^{p}\right] $ that covers $%
R_{p};$ we denote by $\mathcal{H}_{2,p}$ the family of cells of order $p$
that are not contained in any cell of $\mathcal{H}_{1,p},$ that is $\mathcal{%
H}_{2,p}\subset R_{p}^{c}.$ Therefore, when $\Omega _{p,1}$ occurs, the
number of cells in $\mathcal{H}_{2,p}$ is at most $M_{p}$ with $%
2^{-2^{p+1}}M_{p}\leq \frac{4\pi }{2^{p-1}}$, that is 
\begin{equation}
M_{p}\leq 2^{2^{p+1}}\frac{4\pi }{2^{p-1}}.  \label{Num-H2,p}
\end{equation}%
Let $\varepsilon =C_{4,2}(2^{-2^{p}})^{(\alpha /2-1)d}2^{p/2}$ with some
positive constant $C_{4,2}$ to be determined. Define $\Omega _{p,2}$ be the
event that under the condition that 
\begin{equation*}
\sup_{y,y^{\prime }\in \mathcal{V}_{p}(x_{p,i})}\left\Vert \mathbf{T}(y)-%
\mathbf{T}(y^{\prime })\right\Vert \leq C_{4,2}(2^{-2^{p}}2^{p/2})^{\alpha
/2-1},
\end{equation*}%
it also follows that 
\begin{equation*}
L\left( \mathbf{T}(x_{p,i}),D(x_{p,i},2^{-2^{p}})\right) \geq \frac{\pi
(2^{-2^{p}})^{2}}{\varepsilon }
\end{equation*}%
holds for each cell of order $p$ in $\mathbb{S}^{2}$". Applying Lemma \ref%
{LemTupper}, we obtain that for some constants $C_{4,2}>K_{2,7}\sqrt{K_{2,1}}%
,$ the probability $\mathbb{P}\left\{ \Omega _{p,2}^{c}\right\} $ is bounded
by 
\begin{eqnarray*}
&&\sum_{x_{p,j}\in \mathcal{H}_{2,p}\cap \Xi _{p}}\mathbb{P}\left\{ L\left( 
\mathbf{T}(x_{p,j}),D(x_{p,j},2^{-2^{p}})\right) <\frac{\nu
(D(x_{p,j},2^{-2^{p}}))}{\epsilon }\right\} \\
&\leq &\sum_{x_{p,j}\in \mathcal{H}_{2,p}\cap \Xi _{p}}\mathbb{P}\left\{ \mu
\left( B\left( \mathbf{T}(x_{p,j}),\frac{\epsilon }{2}\right) ,D\right) <\nu
(D(x,2^{-2^{p}}))\right\} \\
&\leq &\sum_{x_{p,j}\in \mathcal{H}_{2,p}\cap \Xi _{p}}\mathbb{P}\left\{
\sup_{y,y^{\prime }\in D(x_{p,j},2^{-2^{p}})}\left\Vert \mathbf{T}(y)-%
\mathbf{T}(y^{\prime })\right\Vert >C_{4,2}(2^{-2^{p}})^{\alpha
/2-1}2^{p/2}\right\} \\
&\leq &M_{p}\exp \left\{ -2^{p}\right\} \leq 32\pi \exp \left\{ -2^{p}(1-\ln
2)-p\ln 2\right\} .
\end{eqnarray*}%
Hence we have 
\begin{equation}
\sum_{p=1}^{\infty }\mathbb{P}\left\{ \Omega _{p,2}^{c}\right\} <\infty .
\label{Omega2c}
\end{equation}%
Finally, let $\Omega _{p,3}$ be the event: for any $x\in \mathbb{S}^{2}$ and 
$\mathbf{t\in \mathbb{R}}^{d},$ 
\begin{equation}
\sup_{\left\Vert \mathbf{s}-\mathbf{t}\right\Vert \leq
2K_{4,3}w(2^{-2^{p}})}\left\vert L\left( \mathbf{s},D(x,2^{-2^{p}})\right)
-L\left( \mathbf{t},D(x,2^{-2^{p}})\right) \right\vert \leq \frac{\pi }{%
2K_{4,3}}\phi (2^{-2^{p}}).  \label{def:Omega3}
\end{equation}%
By Lemma \ref{Lem:Local Time}, we have for any even integer $n\geq 2$ 
\begin{eqnarray*}
\mathbb{P}\left\{ \Omega _{p,3}^{c}\right\} &\leq &\frac{%
(K_{2,3})^{n}(n!)^{2-\eta /2}\left\vert 2K_{4,3}w(2^{-2^{p}})\right\vert
^{n\gamma }(2^{-2^{p}})^{(n-1)\eta +2}}{\left\vert \frac{\pi }{2K_{4,3}}\phi
(2^{-2^{p}})\right\vert ^{n}} \\
&\leq &\frac{C_{4,3}(2^{-2^{p}})^{2-\eta -n\gamma (\alpha /2-1)}}{%
p^{n(\alpha /2-1)(d+\gamma )}},
\end{eqnarray*}%
where the constant $C_{4,3}$ is positive and depends on $n,$ $K_{2,3}$ and $%
K_{4,3}.$ Let $n$ be large enough so that $n\gamma \left( \alpha /2-1\right)
\geq 2-\eta ,$ then we obtain 
\begin{equation}
\sum_{p=1}^{\infty }\mathbb{P}\left\{ \left( \Omega _{p,3}\right)
^{c}\right\} <\infty .  \label{Omega3c}
\end{equation}

Now set $\mathcal{H}_{p}=\mathcal{H}_{1,p}\cup \mathcal{H}_{2,p}$. This
family is well-defined for all $p\geq 1,$ and it is a non-overlapping
covering of $\mathbb{S}^{2}.$ Set 
\begin{eqnarray*}
r_{A} &=&2K_{4,3}w\left( 2^{-2^{p}}\right) ,\text{ if }A\in \mathcal{H}_{1,p}%
\text{ and }A\text{ is of order }k\in \left[ p,2p\right] , \\
r_{A} &=&C_{4,2}(2^{-2^{p}})^{(\alpha /2-1)d}2^{p/2},\text{ if }A\in 
\mathcal{H}_{2,p}.
\end{eqnarray*}%
For each cell $A\in \mathcal{H}_{p}$, denote by $\left\vert A\right\vert $
the diameter of $A.$ We pick its center point $x_{A},$ and define $\Omega
_{A}$ the set of events such that 
\begin{equation*}
\left\{ \omega :\left\Vert \mathbf{T}\left( x_{A}\right) -\mathbf{t}%
\right\Vert \leq r_{A}\text{, }L\left( \mathbf{t},A\right) \geq \left\{ 
\begin{array}{cc}
\frac{\pi }{2K_{4,3}}\phi \left( \frac{1}{2}\left\vert A\right\vert \right) ,
& \text{ for }A\in \mathcal{H}_{1,p} \\ 
\frac{\pi }{C_{4,2}2^{\frac{p}{2}}p^{\frac{d}{4}(\alpha -2)}}\phi \left( 
\frac{1}{2}\left\vert A\right\vert \right) , & \text{ for }A\in \mathcal{H}%
_{2,p}%
\end{array}%
\right. \right\} .
\end{equation*}%
Also, denote by $\mathcal{F}_{p}$ the subfamily of $\mathcal{H}_{p}$ such
that 
\begin{equation*}
\mathcal{F}_{p}=\left\{ A\in \mathcal{H}_{p}:\Omega _{A}\text{ occurs}%
\right\} .
\end{equation*}%
Letting $\Omega _{p}=\Omega _{p,1}\cap \Omega _{p,2}\cap \Omega _{p,3},$ we
have 
\begin{equation}
\sum_{p=1}^{\infty }\mathbb{P}\left\{ \Omega _{p}^{c}\right\} <\infty ,
\label{Omegac}
\end{equation}%
in view of (\ref{Omega1c}), (\ref{Omega2c}) and (\ref{Omega3c}). We claim
that 
\begin{equation}
\text{For }p\text{ large enough, on }\Omega _{p}\text{, }\mathcal{F}_{p}%
\text{ covers }\mathbf{T}^{-1}\left( \mathbf{t}\right) .  \label{cover1}
\end{equation}

To see this, consider that for any $x\in \mathbf{T}^{-1}(\mathbf{t}),$ if $x$
belongs to some cell $A$ of order $k$ in $\mathcal{H}_{1,p},$ that is $%
2^{p-1}\leq k\leq 2^{p},$ then recall (\ref{Ineq-H1,p}), we have that, for $%
p $ large enough, 
\begin{equation*}
\left\Vert \mathbf{T}(x_{A})-\mathbf{t}\right\Vert =\left\Vert \mathbf{T}%
(x_{A})-\mathbf{T}(x)\right\Vert \leq r_{A},
\end{equation*}%
and 
\begin{eqnarray*}
L(\mathbf{t},A) &\geq &L\left( \mathbf{T}(x_{A}),A\right) -\left\vert L(%
\mathbf{t},A)-L\left( \mathbf{T}(x_{A}),A\right) \right\vert \\
&\geq &\frac{\pi }{K_{4,3}}\phi (\frac{1}{2}\left\vert A\right\vert )-\frac{%
\pi }{2K_{4,3}}\phi \left( \frac{1}{2}\left\vert A\right\vert \right) \geq 
\frac{\pi }{2K_{4,3}}\phi \left( \frac{1}{2}\left\vert A\right\vert \right) ,
\end{eqnarray*}%
in view of $\left( \ref{def:Omega3}\right) .$ Otherwise, $x$ belongs to some
cell $A$ of order $2^{p}$ in $\mathcal{H}_{2,p},$ and the inequality above
is then readily seen in view of the definition of $\Omega _{p,2}.$ Therefore 
$A\in \mathcal{F}_{p},$ and (\ref{cover1}) is proved.

Now we see that on $\Omega _{p},$ 
\begin{equation*}
\sum_{A\in \mathcal{H}_{1,p}}\phi \left( \left\vert A\right\vert \right)
\leq \pi ^{-1}K_{4,3}2^{3-d\left( \alpha /2-1\right) }\sum_{A\in \mathcal{H}%
_{1,p}}L(\mathbf{t},A),
\end{equation*}%
and 
\begin{equation*}
\sum_{A\in \mathcal{H}_{2,p}}\phi \left( \left\vert A\right\vert \right)
\leq \pi ^{-1}C_{4,2}2^{3-d(\alpha /2-1)+p/2}p^{d(\alpha -2)/4}\sum_{A\in 
\mathcal{H}_{2,p}}L(\mathbf{t},A).
\end{equation*}%
In the meantime, recalling Proposition \ref{Prop:LIL} and the inequality (%
\ref{Num-H2,p}), we have 
\begin{equation*}
\begin{split}
& 2^{p/2}p^{d(\alpha -2)/4}\sum_{A\in \mathcal{H}_{2,p}}L(\mathbf{t},A)\leq
4\pi K_{4,1}2^{p/2}p^{d(\alpha -2)/4}\frac{2^{2^{p+1}}}{2^{p-1}}\phi
(2^{-2^{p}}) \\
& \leq 4\pi K_{4,1}\frac{p^{d(\alpha -2)/4}}{2^{p/2-1}}\left[
2^{2^{p+1}}\phi (2^{-2^{p}})\right] \leq \frac{1}{1-2^{1-p}}\sum_{A\in 
\mathcal{H}_{1,p}}\phi \left( \left\vert A\right\vert \right) ,
\end{split}%
\end{equation*}%
with $p$ large enough such that $4\pi K_{4,1}\frac{p^{d(\alpha -2)/4}}{%
2^{p/2-1}}<1.$ Thus 
\begin{equation*}
\sum_{A\in \mathcal{H}_{p}}\phi \left( \left\vert A\right\vert \right) \leq
K_{4,4}\sum_{A\in \mathcal{H}_{1,p}}L(\mathbf{t},A)\leq K_{4,4}\sum_{A\in 
\mathcal{H}_{p}}L(\mathbf{t},A),
\end{equation*}%
with $K_{4,4}>0$ depending on $\alpha ,d,$ $K_{4,3}$ and $C_{4,2}.$ Hence by
the Borel-Cantelli Lemma, almost surely we have 
\begin{equation*}
\phi \text{-}m\left( \mathbf{T}^{-1}(\mathbf{t})\right) \leq
\liminf_{p\rightarrow \infty }\sum_{A\in \mathcal{H}_{p}}\phi \left(
\left\vert A\right\vert \right) \leq K_{4,4}L\left( \mathbf{t},\mathbb{S}%
^{2}\right) ,
\end{equation*}%
in view of (\ref{Omegac}) and (\ref{cover1}), which completes the proof.
\end{proof}

\bigskip

\begin{proof}[\textbf{Proof of \protect\ref{Th:Measure}\ in Theorem \protect
\ref{Th:Main}}]
The proof follows immediately by combining Theorems \ref{Th:HM-Lower} and %
\ref{Th:HM-Upper}.
\end{proof}

\subsection{Two Examples}

In this subsection, we illustrate Theorem \ref{Th:Main} by two examples.

\begin{example}
Let $H>0$ with $H\neq 1.$ Consider the centered isotropic Gaussian random
field $W_{H}^{0}=\left\{ W_{H}^{0}(x),\ x\in \mathbb{S}^{2}\right\} $ with
covariance 
\begin{equation*}
K_{H}\left( d_{\mathbb{S}^{2}}(x,y)\right) =Cov\left(
W_{H}^{0}(x),W_{H}^{0}(y)\right) =\int_{\mathbb{R}^{+}}\psi ^{(H)}\left( d_{%
\mathbb{S}^{2}}(x,y),r\right) r^{2H-3}dr,
\end{equation*}%
where 
\begin{equation*}
\psi ^{(H)}=\left\{ 
\begin{array}{cl}
\psi & if\ 0<H<1 \\ 
\psi -4\pi & if\ H>1%
\end{array}%
\right.
\end{equation*}%
with 
\begin{equation*}
\psi \left( d_{\mathbb{S}^{2}}(x,y),r\right) =\int_{\mathbb{S}^{2}}\mathrm{1}%
\left\{ d_{\mathbb{S}^{2}}(x,z)<r\right\} \mathrm{1}\left\{ d_{\mathbb{S}%
^{2}}(y,z)<r\right\} d\nu (z).
\end{equation*}%
Notice that if $d_{\mathbb{S}^{2}}(x,y)\geq 2r,\ $then $\psi \left( d_{%
\mathbb{S}^{2}}(x,y),r\right) =0.$ So we will just consider the case $d_{%
\mathbb{S}^{2}}(x,y)\in (0,2r).$ In this circumstance, careful calculations
show that, the covariance $K_{H}\left( d_{\mathbb{S}^{2}}(x,y)\right) $ can
be represented as follows: 
\begin{equation*}
K_{H}\left( d_{\mathbb{S}^{2}}(x,y)\right) =\sum_{\ell =1}^{+\infty }\frac{%
2\ell +1}{4\pi }(C_{\ell }+d_{\ell })P_{\ell }\left( \left\langle
x,y\right\rangle \right) ,
\end{equation*}%
where 
\begin{equation*}
C_{4,4}^{-1}\ell ^{-\left( 2H+2\right) }\leq C_{\ell }\leq C_{4,4}\ell
^{-\left( 2H+2\right) }.
\end{equation*}%
and $d_{2\ell }=0$ with%
\begin{equation*}
C_{4,5}^{-1}\left( 2\ell +1\right) ^{-2}\left( 2\pi \right) ^{-2\ell -1}\leq
d_{2\ell +1}\leq C_{4,5}\left( 2\ell +1\right) ^{-2}\left( 2\pi \right)
^{-2\ell -1},
\end{equation*}%
for $\ell =1,2,....$

It is readily seen that the spherical random field $W_{H}^{0}$ satisfies
Condition(\textbf{A}) with angular power spectrum index $\alpha =2H+2$ if $%
0<H<\frac{1}{2},$ and $\alpha \geq 4$ if $H\geq \frac{1}{2}$ and $H\neq 1.$
Therefore, the level sets of the random field $\mathbf{W}_{H}=\left(
W_{H}^{1},...,W_{H}^{d}\right) $ with $W_{H}^{1},...,W_{H}^{d}$ independent
copies of of $W_{H}^{0},$ is nonempty $a.s.$ if and only if $0<H<1$ and $%
Hd<2.$ Moreover, its exact Hausdorff measure of the level set $\mathbf{W}%
_{H}^{-1}(\mathbf{t})$ at any value $\mathbf{t\in \mathbb{R}}^{d}$ can be
controlled upper and lower by its local time $L\left( \mathbf{t},\mathbb{S}%
^{2}\right) ,$ $a.s.$ with Hausdorff dimension $2-Hd$ for $H\neq \frac{1}{2}$
in view of Theorem \ref{Th:Main} in this paper and Theorem 1.2 in \cite%
{LanXiao}. See \cite{EsIstas10} for more information on the statistical
properties of the random field $W_{H}^{0}$.
\end{example}

\begin{example}
Consider the zero-mean, isotropic Gaussian random field $T_{0}=$ $\left\{
T_{0}(x),x\in \mathbb{S}^{2}\right\} $ with covariance 
\begin{equation*}
C(x,y)=1-\frac{2}{\pi }d_{\mathbb{S}^{2}}(x,y)
\end{equation*}%
for any $x,y\in \mathbb{S}^{2}$ (c.f. \cite{CX16}, Example 3.3). By careful
calculations, we can prove that, 
\begin{equation*}
C(x,y)=C_{1}P_{1}\left( \left\langle x,y\right\rangle \right) +\sum_{\ell
=1}^{+\infty }\left( 4\ell +3\right) C_{\ell }P_{2\ell +1}\left(
\left\langle x,y\right\rangle \right) ,
\end{equation*}%
where for each $\ell >1,$ the angular power spectrum%
\begin{equation*}
C_{\ell }=\sum_{n=\ell }^{\infty }\frac{1}{\left( 2n+2\ell +3\right) !!}%
\frac{\left[ \left( 2n-1\right) !!\right] ^{2}}{\left( 2n-2\ell \right) !!},
\end{equation*}%
and moreover, 
\begin{equation*}
\frac{1}{8\sqrt{e}}\ell ^{-1}\leq C_{\ell }\leq \frac{1}{4}\ell ^{-1}.
\end{equation*}%
That is, its angular power spectrum index $\alpha =1,$ which is less than $2$%
. Thus, the condition (\ref{Specon}) does not hold, which leads to that
Theorem \ref{Th:Main} can not be applied to this spherical random field.
Actually, even the representation (\ref{rapT}) does not hold in the sense of
(\ref{rapT-MSC}) and (\ref{rapT-L2}), as the right-hand side of (\ref%
{rapT-MSC}) and (\ref{rapT-L2}) do not converge.
\end{example}

\section{Proof of Proposition \protect\ref{Smooth} \label{Sec:PropSmooth}}

For any two integers $1\leq L<U\leq \infty ,$ we introduce the band-limited
random field $\mathbf{T}^{L,U}=\left( T_{1}^{L,U},...,T_{d}^{L,U}\right) $
on $\mathbb{S}^{2},$ where $T_{k}^{L,U},k=1,...,d,$ are independent copies
of 
\begin{equation}  \label{defT_LU}
T_{0}^{L,U}( x) =\sum_{\ell =L}^{U}\sum_{m=-\ell }^{\ell }a_{\ell m}Y_{\ell
m}( x) ,\ x\in \mathbb{S}^{2}.
\end{equation}
Observe that $T_{0}^{L,U}( x) $ and $T_{k}^{L^{\prime },U^{\prime}}( x) $
are independent for $L<U<L^{\prime }<U^{\prime }$ in view of the
orthogonality properties of the Fourier components of the field $T_{0}(x)$
and the assumption of Gaussianity. Let $D( z,r) $ be an open disk on $%
\mathbb{S}^{2},$ we start by proving the following lemma.

\begin{lemma}
\label{LemTmain} Under the conditions of Theorem \ref{Th:Main}, there exists
a positive constants $K_{5,1}$ depending on $\alpha ,\ d$ and $K_{0},$ such
that for any $0<r<\delta _{0}$ and $\varepsilon >0,$ we have 
\begin{equation}
\mathbb{P}\left\{ \sup_{x,y\in D(z,r)}\left\Vert \mathbf{T}^{L,U}(x)-\mathbf{%
T}^{L,U}(y)\right\Vert \leq \varepsilon \right\} \geq \exp \left( -K_{5,1}%
\frac{r^{2}}{\varepsilon ^{4/(\alpha -2)}}\right) ,  \label{Tmain}
\end{equation}
\end{lemma}

\begin{proof}
The canonical metric $d_{\mathbf{T}^{a,b}}$ on $\mathbb{S}^{2}$ is defined
by 
\begin{equation*}
d_{\mathbf{T}^{L,U}}( x,y) =\sqrt{\mathbb{E}\left\Vert \mathbf{T}^{L,U}( x) -%
\mathbf{T}^{L,U}( y) \right\Vert ^{2}} =\sqrt{d\mathbb{E}\left\vert
T_{0}^{L,U}( x) -T_{0}^{L,U}(y) \right\vert ^{2}}.
\end{equation*}
By the lower bound in the inequalities $( \ref{variogram}) $ and the fact
that $d_{\mathbf{T}^{L,U}}( x,y) \leq d_{\mathbf{T}}( x,y) $ for any $x,y\in 
\mathbb{S}^{2}$ with $d_{\mathbb{S}^{2}}( x,y) <\delta $, we have 
\begin{equation}  \label{ineq:dTmain}
d_{\mathbf{T}^{L,U}}( x,y) \leq \sqrt{dK_{2,1}}\rho _{\alpha}( x,y) .
\end{equation}
In the meantime, recall the metric entropy, it follows immediately that 
\begin{equation*}
N_{d_{\mathbf{T}^{L,U}}}\left( D( z,r) ,\epsilon \right) \leq C_{5,1}\frac{%
r^{2}}{\epsilon ^{4/( \alpha -2) }}.
\end{equation*}
where $C_{5,1}$ is a positive constant depending on $\alpha ,\ d$ and $%
K_{2,1}.$ Therefore, the approximation $( \ref{Tmain}) $\ is derived by
exploiting Lemma \ref{DudleyLB}.
\end{proof}

Let $\mathbf{T}^{\Delta }$ be the random field defined by 
\begin{equation*}
\mathbf{T}^{\Delta }=\mathbf{T}-\mathbf{T}^{L,U}.
\end{equation*}
Meanwhile, for any $r\in ( 0,\delta _{0}) ,$ take $L=\left[B^{-\beta }r^{-1}%
\right] $ and $U=\left[ B^{1-\beta }r^{-1}\right] ,$ where the constants $%
B,\beta $ satisfy $\beta \in (0,\alpha /2-1]$ and $\max\left\{ 1,\left(
4K_{2,4}\right) ^{\frac{1}{\beta \left( 4-\alpha \right) }}\right\}
<B<r^{-1} $ with $K_{2,4}$ the constant in Lemma \ref{Lem:SumCl}. Here $%
\left[ \cdot \right] $ denotes integer part as usual. Then we have the
following estimate on the tail probability of maximal oscillation of $%
\mathbf{T}^{\Delta }.$

\begin{lemma}
\label{LemTtail} Under the conditions of Theorem \ref{Th:Main}, there exists
a positive constant $K_{5,2}$ depending on $\alpha ,\ d$ and $K_{0},$ such
that for any $0<r<\delta _{0},$ $0<\beta \leq \frac{\alpha }{2}-1$ and 
\begin{equation*}
u > K_{5,2}B^{-\beta \left( 2-\frac{\alpha }{2}\right) }\sqrt{\log B}\ r^{%
\frac{\alpha }{2}-1},
\end{equation*}
we have 
\begin{equation*}
\mathbb{P}\left\{ \sup_{x,y\in D( z,r) }\left\Vert \mathbf{T}^{\Delta }( x) -%
\mathbf{T}^{\Delta }( y) \right\Vert\geq u\right\} \leq \exp \left( -\frac{1%
}{K_{5,2}}\frac{B^{\beta (4-\alpha ) }u^{2}}{r^{\alpha -2}}\right) .
\end{equation*}
\end{lemma}

\begin{proof}
Like in many other arguments in this paper, we start by considering a
suitable Gaussian metric $d_{T_{\Delta }}$ defined on $D( z,r)\subset 
\mathbb{S}^{2}$ by 
\begin{equation*}
d_{\mathbf{T}^{\Delta }}( x,y) :=\left[ \mathbb{E}\left\Vert \mathbf{T}%
^{\Delta }( x) -\mathbf{T}^{\Delta }( y)\right\Vert ^{2}\right] ^{1/2}.
\end{equation*}
Once again, we have $d_{\mathbf{T}^{\Delta }}( x,y) \leq \sqrt{dK_{2,1}}\rho
_{\alpha }( x,y).$ A simple metric entropy argument yields 
\begin{equation*}
N_{d_{\mathbf{T}^{\Delta }}}\left( D( z,r) ,\varepsilon \right) \leq C_{5,1}%
\frac{r^{2}}{\epsilon ^{4/( \alpha -2) }}.
\end{equation*}
More precisely, recall 
\begin{equation*}
d_{\mathbf{T}^{\Delta }}^{2}( x,y) =d\left( \sum_{\ell=0}^{L-1}+\sum_{\ell
=U+1}^{\infty }\right) \frac{2\ell +1}{4\pi }C_{\ell}\left\{ 1-P_{\ell
}(\cos \theta )\right\} .
\end{equation*}
where $\theta =d_{\mathbb{S}^{2}}( x,y).$ Let $0<\theta <r,$ by Lemma \ref%
{Lem:SumCl}, we obtain 
\begin{eqnarray*}
d_{\mathbf{T}^{\Delta }}^{2}( x,y) &\leq & dK_{2,4}\left(L^{4-\alpha }\theta
^{2}+U^{2-\alpha }\right) \\
&\leq & dK_{2,4}\left( B^{-\beta ( 4-\alpha ) }+B^{-( 1-\beta) ( \alpha -2)
}\right) r^{\alpha -2} \\
&\leq & dK_{2,4}B^{-\beta ( 4-\alpha ) }r^{\alpha -2} :=\left\vert f\left(
B\right) \right\vert ^{2}r^{\alpha -2},
\end{eqnarray*}
where 
\begin{equation}  \label{def:f(B)}
f\left( B\right) =:\sqrt{dK_{2,4}}B^{-\frac{\beta }{2}\left( 4-\alpha\right)
}.
\end{equation}
Hence, if we let 
\begin{equation*}
\overline{d}:=\sup \left\{ d_{\mathbf{T}^{\Delta }}( x,y) :x,y\in D( z,r)
\right\} ,
\end{equation*}
which obviously satisfies $\overline{d}\leq f( B) r^{\frac{\alpha }{2}-1}$,
then 
\begin{equation*}
\begin{split}
& \int_{0}^{\overline{d}}\sqrt{\log N_{d_{\mathbf{T}^{\Delta }}}\left(D(
z,r) ,\epsilon \right) }\ d\epsilon \leq \int_{0}^{f(B) r^{\frac{\alpha }{2}%
-1}}\sqrt{\log C_{5,1}\frac{r^{2}}{\epsilon^{4/( \alpha -2) }}}\ d\epsilon \\
& \leq \frac{2}{\sqrt{\alpha -2}}\sqrt{dK_{2,1}}r^{\frac{\alpha }{2}-1}\int_{%
\sqrt{\log \frac{\sqrt{dK_{2,1}}}{f( B) }}}^{+\infty
}ud\left(-e^{-u^{2}}\right) \\
& \leq \frac{4f( B) }{\sqrt{\alpha -2}}\sqrt{\log \frac{\sqrt{dK_{2,1}}}{f(
B) }}\ r^{\frac{\alpha }{2}-1}.
\end{split}%
\end{equation*}
By Lemma \ref{DudleyUB}, we derive that, for any $u>C_{5,2}B^{-\beta
(2-\alpha /2) }\sqrt{\log B}\ r^{\alpha /2-1}$ with some constant $C_{5,2}$
depending on $\alpha $ and $K_{2,7},$ it holds that 
\begin{equation*}
\begin{split}
& \mathbb{P}\left\{ \sup_{x,y\in D( z,r) }\left\Vert \mathbf{T}^{\Delta }(
x) -\mathbf{T}^{\Delta }( y) \right\Vert\geq u\right\} \\
& \leq \mathbb{P}\left\{ \sup_{x,y\in D( z,r) }\left\Vert \mathbf{T}^{\Delta
}( x) -\mathbf{T}^{\Delta }( y) \right\Vert \geq K_{2,7}\bigg(\frac{u}{%
2K_{2,7}}+\int_{0}^{\overline{d}}\sqrt{\log N_{d_{\mathbf{T}^{\Delta
}}}\left( D( z,r) ,\varepsilon \right) }d\varepsilon \bigg)\right\} \\
& \leq \exp \left( -\frac{u^{2}}{4K_{2,7}^{2}\left\vert f( B)\right\vert
^{2}r^{\alpha -2}}\right) .
\end{split}%
\end{equation*}
The proof is then completed.
\end{proof}

Now recalling the representation of local time (\ref{repLT}), we introduce
the following random field for any $D=D\left( x,r\right) \subseteq \mathbb{S}%
^{2},$ $x\in \mathbb{S}^{2}$ 
\begin{equation}  \label{repLT-LU}
L^{L,U}\left( \mathbf{T}\left( x\right) ,D\right) =\frac{1}{2\pi }\int_{%
\mathbb{R}^{d}\times D}\exp \left\{ i\boldsymbol{\xi }^{T}\left( \mathbf{T}%
^{L,U}( y) -\mathbf{T}^{L,U}( x) \right) \right\} d\nu ( y) d\boldsymbol{\xi 
}.
\end{equation}
Obviously, $L^{L,U}\left( \mathbf{T}( x) ,D\right) $ and $L^{L^{\prime
},U^{\prime }}\left( \mathbf{T}( x) ,D\right) $ are independent for $%
L<U<L^{\prime }<U^{\prime }$. We can now prove that

\begin{lemma}
\label{LemLTtail} Under the conditions of Theorem \ref{Th:Main}, there exist
positive constants $K_{5,3}$ and $B_{0}$ depending on $K_{1,0},$ $d$ and $%
\alpha ,$ such that for any $x\in \mathbb{S}^{2}$ fixed, $r\in
\left(0,\delta _{0}\right) $ and $B>B_{0}$, we have for any $A>0,$ 
\begin{equation*}
\mathbb{P}\left\{ \left\vert L\left( \mathbf{T}( x) ,D\right)-L^{L,U}\left( 
\mathbf{T}( x) ,D\right) \right\vert \geq \frac{1}{A}\phi ( r) \right\} \leq
K_{5,3}\frac{A^{2}\cdot B^{-\kappa\beta \left( 4-\alpha \right) }}{\left(
\log \left\vert \log r\right\vert\right) ^{( \alpha -2) d}},
\end{equation*}
where $\kappa =\min \left\{ 2,\frac{4-d( \alpha -2) }{(\alpha -2) }\right\} $
and $\phi ( \cdot ) $ is the function defined in (\ref{def:psi}).
\end{lemma}

\begin{proof}
For every $x\in \mathbb{S}^{2}$ fixed, we introduce random fields $\mathbf{Z}%
_{l}=\left( Z_{l,1},...,Z_{l,d}\right) ,$\ $l=a,b,\Delta ,$ with $Z_{a,k},$\ 
$Z_{b,k},$ and $Z_{\Delta ,k},$ $k=1,...,d,$ independent copies of $Z_{a,0},$%
\ $Z_{b,0},$ and $Z_{\Delta ,0}$ where the random fields $Z_{l,0}=\left\{
Z_{l,0}(y),y\in \mathbb{S}^{2}\right\} ,l=a,b,\Delta ,$ are defined as
follows:%
\begin{eqnarray*}
Z_{a,0}(y) &=&:T_{0}(y)-T_{0}(x),\  \\
Z_{b,0}(y) &=&:T_{0}^{L,U}(y)-T_{0}^{L,U}(x), \\
Z_{\Delta ,0}(y) &=&:Z_{a,0}(y)-Z_{b,0}(x).\ 
\end{eqnarray*}%
Recall the representations (\ref{repLT}) and (\ref{repLT-LU}), we have%
\begin{equation}
\mathbb{E}\left[ L\left( \mathbf{T}(x),D\right) -L^{L,U}(\mathbf{T}(x),D)%
\right] ^{2}=\frac{1}{(2\pi )^{2}}\int_{D^{2}}II(y_{1},y_{2})d\nu \left(
y_{1}\right) d\nu \left( y_{2}\right) ,  \label{Eq:int-II}
\end{equation}%
where 
\begin{equation*}
II(y_{1},y_{2})=\int_{\mathbb{R}^{2d}}\prod\limits_{j=1,2}\left( \mathbb{E}%
\exp \left\{ i\sum\limits_{k=1}^{d}\xi _{k,j}Z_{a,k}(y_{j})\right\} -\mathbb{%
E}\exp \left\{ i\sum\limits_{k=1}^{d}\xi _{k,j}Z_{b,k}(y_{j})\right\}
\right) d\boldsymbol{\xi .}
\end{equation*}%
Now let us focus on the term $II(y_{1},y_{2}).$ For any $y_{1},y_{2}\in 
\mathbb{S}^{2},$ denote by $\sigma _{l,i}^{2}=:\mathrm{Var}\left(
Z_{l,0}(y_{i})\right) ,\ l=a,b,\Delta ,\ i=1,2,$ and $\varrho _{l}=:\mathrm{%
Cov}\left( Z_{l,0}\left( y_{1}\right) ,Z_{l,0}\left( y_{2}\right) \right)
,l=a,b,\Delta .$ Let $\Upsilon _{l}$ be covariance matrices of $\left(
Z_{l,0}(y_{1}),Z_{l,0}(y_{2})\right) ,$ that is 
\begin{equation*}
\Upsilon _{l}=:\left[ 
\begin{array}{cc}
\sigma _{l,1}^{2} & \varrho _{l} \\ 
\varrho _{l} & \sigma _{l,2}^{2}%
\end{array}%
\right]
\end{equation*}%
for $l=a,b,\Delta .$ In the meantime, denote by 
\begin{equation*}
\widetilde{\Upsilon }_{l}=:\left[ 
\begin{array}{cc}
\sigma _{l,1}^{2} & \varrho _{b} \\ 
\varrho _{b} & \sigma _{l^{\prime },2}^{2}%
\end{array}%
\right] ,l,l^{\prime }=a,b,\ and\ l\neq l^{\prime },
\end{equation*}%
then a standard calculation yields that 
\begin{equation}
II(y_{1},y_{2})=\sum\limits_{l=a,b}\frac{1}{(\det \Upsilon _{l})^{d/2}}%
-\sum\limits_{l=a,b}\frac{1}{\left( \det \widetilde{\Upsilon }_{l}\right)
^{d/2}}  \notag
\end{equation}%
Notice that $\Upsilon _{a}=\Upsilon _{b}+\Upsilon _{\Delta },\ \widetilde{%
\Upsilon }_{a}=\Upsilon _{b}+diag\left( \sigma _{\Delta ,1}^{2},0\right) $
and $\widetilde{\Upsilon }_{b}=\Upsilon _{b}+diag\left( 0,\sigma _{\Delta
,2}^{2}\right) .$ Moreover, careful calculations yield 
\begin{equation*}
\det \Upsilon _{a}\geq \det \Upsilon _{b}+\det \Upsilon _{\Delta
}+\left\vert \sigma _{b,1}\sigma _{\Delta ,2}-\sigma _{\Delta ,1}\sigma
_{b,2}\right\vert ^{2}\geq \det \Upsilon _{b}
\end{equation*}%
\begin{equation*}
\det \widetilde{\Upsilon }_{a}=\det \Upsilon _{b}+\sigma _{\Delta
,1}^{2}\sigma _{b,2}^{2},\ and\ \det \widetilde{\Upsilon }_{b}=\det \Upsilon
_{b}+\sigma _{b,1}^{2}\sigma _{\Delta ,2}^{2}.
\end{equation*}%
Therefore, 
\begin{eqnarray}
II(y_{1},y_{2}) &\leq &\frac{2}{\left( \det \Upsilon _{b}\right) ^{d/2}}-%
\frac{1}{\left( \det \Upsilon _{b}+\sigma _{\Delta ,1}^{2}\sigma
_{b,2}^{2}\right) ^{d/2}}  \label{ineq:detR} \\
&&-\frac{1}{\left( \det \Upsilon _{b}+\sigma _{b,1}^{2}\sigma _{\Delta
,2}^{2}\right) ^{d/2}}.  \notag
\end{eqnarray}%
Now let $\theta _{i}=d_{\mathbb{S}^{2}}(y_{i},x),i=1,2,$ then 
\begin{equation*}
\sigma _{\Delta ,i}^{2}=\left( \sum_{\ell =1}^{L-1}+\sum_{\ell =U+1}^{\infty
}\right) \frac{2\ell +1}{4\pi }C_{\ell }\left\{ 1-P_{\ell }(\cos \theta
_{i})\right\} ,
\end{equation*}%
whence by exploiting Lemmas \ref{Lem:SumCl}, we obtain that 
\begin{eqnarray}
\sigma _{\Delta ,i}^{2} &\leq &K_{2,4}\left( L^{4-\alpha }\theta
_{i}^{2}+U^{2-\alpha }\right) \leq K_{2,4}\left( B^{-\beta (4-\alpha
)}+B^{-(1-\beta )(\alpha -2)}\right) \theta _{i}^{\alpha -2}  \notag \\
&\leq &2d^{-1}\left\vert f(B)\right\vert ^{2}\theta _{i}^{\alpha -2}.
\label{ineq:detRd}
\end{eqnarray}%
in view of $(\ref{def:f(B)}).$ Let $B>B_{0}$ such that $\left\vert f\left(
B_{0}\right) \right\vert ^{2}<\frac{d}{4}K_{2,1}^{-1}$ as well as $%
L=L(B_{0},r),U=U(B_{0},r)$ large enough so that $\left\vert \varrho
_{b}\right\vert \leq \frac{1}{2}\left\vert \varrho _{a}\right\vert ,$ then 
\begin{equation}
\sigma _{b,i}^{2}=\sigma _{a,i}^{2}-\sigma _{\Delta ,i}^{2}\geq \frac{1}{2}%
\sigma _{a,i}^{2}\geq \frac{1}{2}K_{2,1}^{-1}\theta _{i}^{\alpha -2},\ i=1,2,
\label{ineq:detRb}
\end{equation}%
which leads to 
\begin{equation}
\det \Upsilon _{b}=\sigma _{b,1}^{2}\sigma _{b,2}^{2}-(\varrho _{b})^{2}\geq 
\frac{1}{4}\left[ \sigma _{a,1}^{2}\sigma _{a,2}^{2}-(\varrho _{a})^{2}%
\right] =\frac{1}{4}\det \Upsilon _{a}.  \label{ineq:detRb-L}
\end{equation}%
Thus, combining inequalities (\ref{ineq:detR}), (\ref{ineq:detRd}), (\ref%
{ineq:detRb}) and (\ref{ineq:detRb-L}), we obtain 
\begin{eqnarray}
II\left( y_{1},y_{2}\right) &\leq &\frac{2}{\left( \det \Upsilon _{b}\right)
^{d/2}}-\frac{2}{\left( \det \Upsilon _{b}+2d^{-1}\left\vert f(B)\right\vert
^{2}\theta _{1}^{\alpha -2}\theta _{2}^{\alpha -2}\right) ^{d/2}}  \notag \\
&\leq &\frac{4d^{-1}C_{5,3}\left\vert f(B)\right\vert ^{2}\theta
_{1}^{\alpha -2}\theta _{2}^{\alpha -2}}{\det \Upsilon _{b}\left( \det
\Upsilon _{b}+2d^{-1}\left\vert f(B)\right\vert ^{2}\theta _{1}^{\alpha
-2}\theta _{2}^{\alpha -2}\right) ^{d/2}}  \notag \\
&\leq &\frac{2^{d+4}d^{-1}C_{5,3}\left\vert f(B)\right\vert ^{2}\theta
_{1}^{\alpha -2}\theta _{2}^{\alpha -2}}{\det \Upsilon _{a}\left( \det
\Upsilon _{a}+8d^{-1}\left\vert f(B)\right\vert ^{2}\theta _{1}^{\alpha
-2}\theta _{2}^{\alpha -2}\right) ^{d/2}}  \label{ineq:II}
\end{eqnarray}%
where the last but one inequality used the fact that $(1+x)^{d/2}\leq
1+C_{5,3}x$ for any $x\in (0,\delta _{0})$ with $C_{5,3}=\frac{d}{2}\max
\left\{ (1+\delta _{0})^{d/2-1},1\right\} .$ Recall (\ref{SLND}) in Lemma %
\ref{Lem:C1C2} and Corollary \ref{C2'}, we have 
\begin{eqnarray}
\det \Upsilon _{a} &=&\mathrm{Var}\left( Z_{a,0}(y_{1})\right) \mathrm{Var}%
\left( Z_{a,0}\left( y_{2}\right) |Z_{a}\left( y_{1}\right) \right)  \notag
\\
&\geq &K_{2,1}^{-1}K_{2,2}^{\prime }\left[ \rho _{\alpha }(\theta _{1})%
\right] ^{2}\min \left( \rho _{\alpha }^{2}(\theta _{2}),\rho _{\alpha
}^{2}\left( d_{\mathbb{S}^{2}}(y_{2},y_{1})\right) \right) .
\label{ineq:detRa}
\end{eqnarray}%
Replacing $\det \Upsilon _{a}$ in the inequality $\left( \ref{ineq:II}%
\right) $\ with the term on the right side of inequality $\left( \ref%
{ineq:detRa}\right) $ above, we obtain 
\begin{equation}
II(y_{1},y_{2})\leq C_{5,4}\frac{\left\vert f(B)\right\vert ^{2}}{\theta
_{1}^{\frac{d}{2}(\alpha -2)}}\cdot A(x_{2})  \label{ineq:II-1}
\end{equation}%
where $C_{5,4}$ is a positive constant depending on $\alpha
,d,C_{5,3},K_{2,1},\ K_{2,2},$ and 
\begin{eqnarray*}
A(x_{2}) &=&\theta _{2}^{\alpha -2}\left[ \min \left( \theta _{2}^{\alpha
-2},d_{\mathbb{S}^{2}}^{\alpha -2}(x_{2},x_{1})\right) \right] ^{-1} \\
&&\cdot \left\{ \min \left( \theta _{2}^{\alpha -2},d_{\mathbb{S}%
^{2}}^{\alpha -2}(x_{2},x_{1})\right) +8d^{-1}\left\vert f(B)\right\vert
^{2}\theta _{2}^{\alpha -2}\right\} ^{-d/2}.
\end{eqnarray*}%
Similar to the argument in the proof of Lemma 3.2 in \cite{LanXiao},\ we
define the following two sets that are disjoint except on the boundary, 
\begin{equation*}
\Gamma _{\tau }=\left\{ y\in D:d_{\mathbb{S}^{2}}(y,\tau )=\min \left\{ d_{%
\mathbb{S}^{2}}(y,\tau ^{\prime }),\tau ^{\prime }=x,y_{1}\right\} \right\}
\end{equation*}%
for $\tau =x,\ y_{1}.$ Recall $\theta _{i}=d_{\mathbb{S}%
^{2}}(y_{i},x),i=1.2, $ then it follows that 
\begin{equation*}
\begin{split}
& \int_{\Gamma _{x}}A(x_{2})d\nu (x_{2})\leq \int_{0}^{2\pi
}\int_{0}^{r(\phi )}\theta _{2}^{-(\alpha -2)d/2}\sin \theta _{2}d\theta
_{2}d\phi \\
& \leq \frac{2}{4-(\alpha -2)d}\int_{0}^{2\pi }\left[ r(\phi )\right] ^{%
\frac{1}{2}\left[ 4-(\alpha -2)d\right] }d\phi \leq C_{5,5}r^{\frac{1}{2}%
\left[ 4-(\alpha -2)d\right] }
\end{split}%
\end{equation*}%
where $C_{5,5}=\frac{4\pi }{4-\left( \alpha -2\right) d},$ and hence, 
\begin{eqnarray}
&&\left\vert f\left( B\right) \right\vert ^{2}\int_{D}\int_{\Gamma _{x}}%
\frac{A\left( y_{2}\right) }{\theta _{1}^{d\left( \alpha -2\right) /2}}d\nu
\left( y_{2}\right) d\nu \left( y_{1}\right)  \notag \\
&\leq &2\pi C_{5,5}\left\vert f\left( B\right) \right\vert ^{2}r^{\frac{1}{2}%
\left[ 4-\left( \alpha -2\right) d\right] }\int_{0}^{r}\theta _{1}^{-\frac{d%
}{2}\left( \alpha -2\right) }\sin \theta _{1}d\theta _{1}  \notag \\
&\leq &C_{5,5}\left\vert f\left( B\right) \right\vert ^{2}r^{4-\left( \alpha
-2\right) d}.  \label{ineq:II'-1}
\end{eqnarray}%
For $\Gamma _{y_{1}},$ we split it into two domains, 
\begin{equation*}
\Gamma _{y_{1}}^{1}=\left\{ y\in \Gamma _{y_{1}}:d_{\mathbb{S}^{2}}\left(
y,y_{1}\right) \leq \left( 8d^{-1}\left\vert f\left( B\right) \right\vert
^{2}\right) ^{1/\left( \alpha -2\right) }\theta _{2}\right\} ,
\end{equation*}%
and 
\begin{equation*}
\Gamma _{y_{1}}^{2}=\left\{ y\in \Gamma _{y_{1}}:d_{\mathbb{S}^{2}}\left(
y,x_{1}\right) >\left( 8d^{-1}\left\vert f\left( B\right) \right\vert
^{2}\right) ^{1/\left( \alpha -2\right) }\theta _{2}\right\} .
\end{equation*}%
Denote by $\theta _{12}=d_{\mathbb{S}^{2}}\left( y_{2},y_{1}\right) $ and
let $\left\vert f(B)\right\vert ^{2}\leq d2^{-\alpha -1}$ as well as $%
B>B_{0},$ then by the triangle inequality 
\begin{equation*}
\theta _{1}-\theta _{12}\leq \theta _{2}\leq \theta _{1}+\theta _{12},
\end{equation*}%
we have $\frac{1}{2}\theta _{1}\leq \theta _{2}\leq 2\theta _{1}$ for $%
y_{2}\in \Gamma _{y_{1}}^{1}.$ Hence, 
\begin{eqnarray*}
&&\int_{\Gamma _{y_{1}}^{1}}A(y_{2})d\nu (y_{2}) \\
&\leq &\frac{2\pi (2\theta _{1})^{\alpha -2}}{\left\{ 8d^{-1}\left\vert
f(B)\right\vert ^{2}(\theta _{1}/2)^{\alpha -2}\right\} ^{d/2}}%
\int_{0}^{2\left( 8d^{-1}\left\vert f(B)\right\vert ^{2}\right) ^{1/(\alpha
-2)}\theta _{1}}\theta _{12}^{2-\alpha }\sin \theta _{12}d\theta _{12} \\
&\leq &C_{5,6}\left\vert f(B)\right\vert ^{\frac{4-d(\alpha -2)}{(\alpha -2)}%
-2}\theta _{1}^{\frac{1}{2}\left[ 4-(\alpha -2)d\right] },
\end{eqnarray*}%
with $C_{5,6}$ a positive constant depending on $\alpha ,d$, which leads to 
\begin{eqnarray}
&&\left\vert f(B)\right\vert ^{2}\int_{D}\int_{\Gamma _{y_{1}}^{1}}\frac{%
A(y_{2})}{\theta _{1}^{d(\alpha -2)/2}}d\nu (y_{2})d\nu (y_{1})  \notag \\
&\leq &C_{5,6}\left\vert f(B)\right\vert ^{\frac{4-d(\alpha -2)}{(\alpha -2)}%
}\int_{0}^{r}\theta _{1}^{2-d(\alpha -2)}\sin \theta _{1}d\theta _{1}  \notag
\\
&\leq &\frac{C_{5,6}}{4-d(\alpha -2)}\left\vert f(B)\right\vert ^{\frac{%
4-d(\alpha -2)}{(\alpha -2)}}r^{4-d(\alpha -2)}.  \label{ineq:II'-2}
\end{eqnarray}%
Meanwhile, 
\begin{eqnarray*}
&&\int_{\Gamma _{y_{1}}^{2}}A(y_{2})d\nu (y_{2}) \\
&\leq &r^{\alpha -2}\int_{0}^{2\pi }\int_{2\left( 8d^{-1}\left\vert
f(B)\right\vert ^{2}\right) ^{1/(\alpha -2)}\theta _{1}}^{r(\phi )}\theta
_{12}^{-(\alpha -2)(d/2+1)}\sin \theta _{12}d\theta _{12}d\phi \\
&\leq &r^{\alpha -2}\int_{0}^{2\pi }\left[ \left[ r(\phi )\right]
^{2-(\alpha -2)(d/2+1)}-C_{5,7}\left\vert f(B)\right\vert ^{\frac{4-(\alpha
-2)(d+2)}{(\alpha -2)}}\right] d\phi \\
&\leq &2\pi \left( r^{\frac{1}{2}\left[ 4-(\alpha -2)d\right]
}-C_{5,7}r^{\alpha -2}\left\vert f(B)\right\vert ^{\frac{4-d(\alpha -2)}{%
(\alpha -2)}-2}\right) ,
\end{eqnarray*}%
with $C_{5,7}$ a positive constant depending on $\alpha ,d,$ and hence 
\begin{eqnarray}
&&\left\vert f(B)\right\vert ^{2}\int_{D}\int_{\Gamma _{y_{1}}^{2}}\frac{%
A(y_{2})}{\theta _{1}^{d(\alpha -2)/2}}d\nu (y_{2})d\nu (y_{1})  \notag \\
&\leq &4\pi ^{2}\left\vert f(B)\right\vert ^{2}r^{\frac{1}{2}\left[
4-(\alpha -2)d\right] }\int_{0}^{r}\theta _{1}^{-d(\alpha -2)}\sin \theta
_{1}d\theta _{1}  \notag \\
&\leq &4\pi ^{2}\left\vert f(B)\right\vert ^{2}r^{4-d(\alpha -2)}.
\label{ineq:II'-3}
\end{eqnarray}%
Recall $f(B)=\sqrt{dK_{2,4}}B^{-\beta (2-\alpha /2)}$ and let $\kappa =\min
\left\{ 2,\frac{4-d(\alpha -2)}{(\alpha -2)}\right\} ,$ then combining the
inequalities $(\ref{ineq:II'-1}),$ $(\ref{ineq:II'-2})$ and $(\ref%
{ineq:II'-3})$ together with $(\ref{Eq:int-II})$ and $(\ref{ineq:II-1})$, we
have 
\begin{equation*}
\mathbb{E}\left[ L\left( \mathbf{T}(x),D\right) -L^{L,U}(\mathbf{T}(x),D)%
\right] ^{2}\leq K_{5,3}B^{-\kappa \beta (4-\alpha )}r^{4-d(\alpha -2)},
\end{equation*}%
where $K_{5,3}>0$ depends on $C_{5,4},C_{5,6},\alpha $ and $d.$ The proof is
then completed in view of the Chebyshev's inequality.
\end{proof}

Recall $w( r) =\rho _{\alpha }\left( r/\sqrt{\log \left\vert \log
r\right\vert }\right) ;$ we are now ready to prove Proposition \ref{Smooth}.

\begin{proof}[\textbf{Proof of Proposition \protect\ref{Smooth}}]
To establish the result, we only need to carefully construct a sequence $%
\left\{ r_{k}\right\} _{k=1}^{k_{0}}\subset \left[ r_{0}^{2},r_{0}\right] ,$
such that there exists a positive constant $C_{5,8}$ such that 
\begin{equation}  \label{PrOutLT}
\begin{split}
&\mathbb{P}\left\{ 
\begin{array}{c}
\text{for every }r_{k}\in \left[ r_{0}^{2},r_{0}\right] ,\ L\left( \mathbf{T}%
( x) ,D( x,r_{k}) \right) \leq \frac{\pi }{C_{5,8}}\phi \left( r_{k}\right)
\\ 
\text{ or }\sup_{d_{\mathbb{S}^{2}}( x,y) \leq r_{k}}\left\Vert \mathbf{T}(
x) -\mathbf{T}( y) \right\Vert>2C_{5,8}w( r_{k})%
\end{array}
\right\} \\
&\leq \exp \left\{ -\left( \left\vert \log r_{0}\right\vert \right) ^{\frac{1%
}{2}}\right\} .
\end{split}%
\end{equation}
Let $r_{k}=r_{0}B^{-\left( k-1\right) },k=1,\cdots ,k_{0},$ for some
constants $k_{0}$ and $B>1$ whose values will be determined later. Moreover, 
$k_{0}$ satisfies $k_{0}\leq \frac{\left\vert \log r_{0}\right\vert }{\log B}%
,$ so that $r_{0}^{2}\leq r_{k}\leq r_{0}$ for all $1\leq k\leq k_{0}.$ Now
let $L_{k}=\left[ \frac{B^{k}}{r_{0}}\right] +1,$ $U_{k}=\left[ \frac{B^{k+1}%
}{r_{0}}\right] ,$ where $\left[ \cdot \right] $ denotes the integer part as
before. Then by Lemma \ref{LemTmain}, we have 
\begin{eqnarray*}
&&\mathbb{P}\left\{ \sup_{d_{\mathbb{S}^{2}}( x,y) \leq r_{k}}\left\Vert 
\mathbf{T}^{L_{k},U_{k}}( x) -\mathbf{T}^{L_{k},U_{k}}( y) \right\Vert \leq
C_{5,8}w( r_{k}) \right\} \\
&\geq &\exp \left( -K_{5,1}( C_{5,8}) ^{-4/( \alpha-2) }\log \left\vert \log
r_{k}\right\vert \right) \geq \left(\left\vert \log r_{k}\right\vert \right)
^{-\frac{1}{4}},
\end{eqnarray*}%
where we have chosen $C_{5,8}$ large enough such that $K_{5,1}\left(C_{5,8}%
\right) ^{-4/\left( \alpha -2\right) }\leq \frac{1}{4}.$

On the other hand, recall $\mathbf{T}^{\Delta _{k}}=\mathbf{T}-\mathbf{T}%
^{L_{k},U_{k}},$ and $2-\alpha /2>\alpha /2-1$ for $2<\alpha <4,$ so we can
let $B$ large enough such that 
\begin{equation*}
\frac{B^{\beta (2-\alpha /2)}}{\sqrt{\log B}}\geq \left( \log \left\vert
\log r_{0}\right\vert \right) ^{\frac{\alpha }{2}-1},
\end{equation*}%
which leads to that, for any $k\in \left\{ 1,...,k_{0}\right\} ,$ 
\begin{equation*}
\rho _{\alpha }\left( r_{k}/\sqrt{\log \left\vert \log r_{k}\right\vert }%
\right) >B^{-\beta (2-\alpha /2)}\sqrt{\log B}r_{k}^{\alpha /2-1},
\end{equation*}%
and 
\begin{eqnarray*}
&&\mathbb{P}\left\{ \sup_{d_{\mathbb{S}^{2}}(x,y)\leq r_{k}}\left\Vert 
\mathbf{T}^{L_{k},U_{k}}\left( x\right) -\mathbf{T}^{L_{k},U_{k}}\left(
y\right) \right\Vert >C_{5,8}w(r_{k})\right\} \\
&\leq &\exp \left( -C_{5,9}\frac{B^{\beta (4-\alpha )}}{\left( \log
\left\vert \log r_{k}\right\vert \right) ^{\alpha /2-1}}\right) ,
\end{eqnarray*}%
where $C_{5,9}=C_{5,8}^{2}/K_{5,2}$ in view of Lemma \ref{LemTtail}.

Meanwhile, for any $\varepsilon >0$, let $B=B\left( \mathbf{T}%
^{L_{k},U_{k}}( x) ,\frac{\varepsilon }{2}\right) $ be an open cube on $%
\mathbb{R}^{d},$ which is centered at $\mathbf{T}^{L_{k},U_{k}}\left(
x\right) $ with length $\epsilon .$ The occupation measure of $\mathbf{T}%
^{L_{k},U_{k}}$ in $B$ can be defined as 
\begin{equation*}
\mu ^{L_{k},U_{k}}( I,D) =\int_{D( x,r) }\mathbf{1}_{B}\left( \mathbf{T}%
^{L_{k},U_{k}}( y) \right) d\nu (y) =\int_{B}L^{L_{k},U_{k}}( \mathbf{t},D(
x,r_{k})) d\mathbf{t},
\end{equation*}
where $L^{L_{k},U_{k}}$ is defined in $( \ref{repLT-LU}) .$ Now let $%
\varepsilon =C_{5,8}w( r) ,$ then by exploiting the continuity of $L\left( 
\mathbf{t},D\right) $ and Lemma \ref{LemTmain}, we have 
\begin{eqnarray}
&&\mathbb{P}\left\{ L^{L_{k},U_{k}}( \mathbf{T}( x),D) <\frac{v_{\mathbb{S}%
^{2}}( D( x,r_{k}) ) }{C_{5,8}w( r_{k}) }\right\}  \notag \\
&\leq & \mathbb{P}\left\{ \mu ^{L_{k},U_{k}}( I,D) <v_{\mathbb{S}^{2}}( D(
x,r_{k}) ) \right\}  \notag \\
&\leq &\mathbb{P}\left\{ \sup_{y\in D( x,r_{k}) }\left\Vert \mathbf{T}%
^{L_{k},U_{k}}( y) -\mathbf{T}^{L_{k},U_{k}}(x) \right\Vert >C_{5,8}w(
r_{k}) \right\}  \notag \\
&\leq &1-\exp \left( -K_{5,1}\frac{r_{k}^{2}}{( C_{5,8}w(r_{k}) ) ^{4/(
\alpha -2) }}\right)  \notag \\
&=&1-\exp \left( -K_{5,1}( C_{5,8}) ^{-4/( \alpha -2)}\log \left\vert \log
r_{k}\right\vert \right) .  \label{PrLab}
\end{eqnarray}
Let the positive constant $C_{5,8}$ be large enough so that $(C_{5,8}) ^{4/(
\alpha -2) }/K_{5,1}\geq 4,$ then 
\begin{equation*}
\mathbb{P}\left\{ L^{L_{k},U_{k}}\left( \mathbf{T}\left( x\right)
,D\left(x,r_{k}\right) \right) \leq \frac{\pi }{C_{5,8}}\phi \left(
r_{k}\right)\right\} \leq 1-\left\vert \log r_{k}\right\vert ^{-\frac{1}{4}}.
\end{equation*}
in view of the inequality $\left( \ref{PrLab}\right) .$ Now define $\mathcal{%
F}_{k,i},i=1,...,4,$ to be the events such that 
\begin{eqnarray*}
\mathcal{F}_{k,1} &=&:\left\{ \omega :\sup_{d_{\mathbb{S}^{2}}\left(x,y%
\right) \leq r_{k}}\left\Vert \mathbf{T}^{L_{k},U_{k}}( y) -\mathbf{T}%
^{L_{k},U_{k}}( x) \right\Vert >C_{5,8}w(r_{k}) \right\} , \\
\mathcal{F}_{k,2} &=&:\left\{ \omega :\sup_{d_{\mathbb{S}^{2}}(x,y) \leq
r_{k}}\left\vert \mathbf{T}^{\Delta _{k}}( x) -\mathbf{T}^{\Delta _{k}}( y)
\right\vert >C_{5,8}w(r_{k}) \right\} , \\
\mathcal{F}_{k,3} &=&:\left\{ \omega :L^{L_{k},U_{k}}\left( \mathbf{T}(x)
,D( x,r_{k}) \right) \leq \frac{2\pi }{C_{5,8}}\phi( r_{k}) \right\} , \\
\mathcal{F}_{k,4} &=&:\left\{ \omega :\left\vert L\left( \mathbf{T}(x) ,D(
x,r_{k}) \right) -L^{L_{k},U_{k}}\left( \mathbf{T}( x) ,D( x,r_{k}) \right)
\right\vert \leq \frac{\pi }{C_{5,8}}\phi ( r_{k}) \right\} .
\end{eqnarray*}
Obviously $\mathcal{F}_{k,3}\subset \mathcal{F}_{k,1}$ according to the
discussion above in (\ref{PrLab}). Therefore, we obtain that 
\begin{eqnarray*}
\left( \ref{PrOutLT}\right) &\leq &\mathbb{P}\left\{ \text{for every }k\in
\left\{ 1,...,k_{0}\right\} ,\mathcal{F}_{k,1}\cup \mathcal{F}_{k,2}\text{
or }\left\{ \mathcal{F}_{k,3}\cap \mathcal{F}_{k,4}\right\} \cup \mathcal{F}%
_{k,4}^{c}\text{ occurs}\right\} \\
&\leq &\prod\limits_{k=0}^{k_{0}}\mathbb{P}\left\{ \mathcal{F}_{k,1}\text{
or }\mathcal{F}_{k,3}\text{ occurs}\right\} + \sum_{k=0}^{k_{0}}\mathbb{P}%
\left\{ \mathcal{F}_{k,2}\text{ occurs}\right\} + \sum_{k=0}^{k_{0}}\mathbb{P%
}\left\{ \mathcal{F}_{k,4}^{c}\text{ occurs}\right\} \\
&\leq &\left( 1-\left\vert \log r_{0}\right\vert ^{-\frac{1}{4}%
}\right)^{k_{0}} +k_{0}\exp \left\{ -C_{5,9}\frac{B^{\beta \left( 4-\alpha
\right) }}{\left( \log \left\vert \log r_{0}\right\vert \right) ^{\alpha
/2-1}}\right\} \\
&&+\pi ^{-2}K_{5,3}C_{5,8}^{2}k_{0}\frac{B^{-\kappa \beta ( 4-\alpha) }}{%
\left( \log \left\vert \log r\right\vert \right) ^{\alpha -2}}.
\end{eqnarray*}
Recall that $k_{0}\leq \left\vert \log r_{0}\right\vert /\log B.$ Let $B$ be
large enough such that 
\begin{equation*}
\min \left\{ K_{5,3}C_{5,8}^{2}/\pi ^{2},C_{5,9}\right\} B^{\min
\left\{\kappa ,1\right\} \beta \left( 4-\alpha \right) } \geq \left\vert
\log r_{0}\right\vert ^{3}
\end{equation*}
as well as $B>B_{0},$ then it is readily seen that for $r_{0}$ small enough, 
\begin{eqnarray*}
\pi ^{-2}K_{5,3}C_{5,8}^{2}k_{0}\frac{B^{-\kappa \beta ( 4-\alpha) }}{\left(
\log \left\vert \log r\right\vert \right) ^{\alpha -2}} &\leq & \frac{%
\left\vert \log r_{0}\right\vert }{\log \left\vert \log r_{0}\right\vert }%
\frac{\left( \left\vert \log r_{0}\right\vert \right) ^{-3}}{\left( \log
\left\vert \log r_{0}\right\vert \right) ^{\alpha -2}} \\
&\leq &\frac{1}{3}\left\vert \log r_{0}\right\vert ^{-2},
\end{eqnarray*}
\begin{equation*}
k_{0}\exp \left\{ -C_{5,9}\frac{B^{\kappa \beta \left( 4-\alpha \right) }}{%
\left( \log \left\vert \log r_{0}\right\vert \right) ^{\alpha /2-1}}\right\}
\leq \frac{1}{3}\left\vert \log r_{0}\right\vert ^{-2}
\end{equation*}
and 
\begin{equation*}
\left( 1-\left\vert \log r_{0}\right\vert ^{-\frac{1}{4}}\right)^{k_{0}}\leq
\exp \left\{ -k_{0}\left\vert \log r_{0}\right\vert ^{-\frac{1}{4}}\right\}
\leq \frac{1}{3}\left\vert \log r_{0}\right\vert ^{-2}.
\end{equation*}
Therefore, 
\begin{equation*}
\left( \ref{PrOutLT}\right) \leq \left\vert \log r_{0}\right\vert ^{-2},
\end{equation*}
which leads to the conclusion of Proposition \ref{Smooth}.
\end{proof}


\begin{thebibliography}{99}
\bibitem{Adler81} Adler, R. J. (1981), \textit{The Geometry of Random Fields}%
, Wiley, New York.

\bibitem{AdlerTaylor} Adler, R. J. and Taylor, J. E. (2007), \emph{Random
Fields and Geometry,} Springer-Verlag.

\bibitem{AdlerT11} Adler, R. J. and Taylor, J. E. (2011), Topological
Complexity of Smooth Random Functions, \emph{Lecture Notes in Mathematics,} 
\textbf{2019}, \'{E}cole d'\'{E}t\'{e} de Probabilit\'{e}s de Saint-Flour.
Springer, Heidelberg.

\bibitem{AdlerTW12} Adler, R. J., Taylor, J. E. and Worsley, K. J. (2017), 
\textit{Applications of Random Fields and Geometry: Foundations and Case
Studies.} \textit{In preparation}.

\bibitem{AW09} J.-M. Aza\"{\i}s and M. Wschebor (2009), \emph{Level Sets and
Extrema of Random Processes and Fields.} John Wiley \& Sons, Hoboken, NJ.

\bibitem{BierLacXiao} Bierm\'{e}, H., Lacaux, C. and Xiao, Y. (2009),
Hitting probabilities and the Hausdorff dimension of the inverse images of
anisotropic Gaussian random fields, \emph{Bull. London Math. Soc.} \textbf{41%
}, 253-273

\bibitem{BarakMount} Baraka, D. and Mountford, T. S. (2011), The exact
Hausdorff measure of the zero set of fractional Brownian motion, \emph{J.
Theor. Probab}, \textbf{24}, 271--293.

\bibitem{Bert} Bertoin, J (1999), Subordinators: examples and applications,
Ecole d'\'{e}t\'{e} de Probabilit\'{e} de St.-Flour XXVII. \emph{Lecture
Notes in Math. }\textbf{1717}, 1--91. Springer, Berlin.

\bibitem{CX16} Cheng, D. and Xiao, Y. (2016), Excursion probability of
Gaussian random fields on sphere. \textit{Bernoulli} \textbf{22}, 1113--1130.

\bibitem{Cuzick82} Cuzick, J. and Du Peeze, J. (1982), Joint continuity of
Gaussian local times,\emph{\ Ann. of Probab.} \textbf{10}, 810--817.

\bibitem{DalangMuellerXiao} Dalang, R.C., Mueller, C. and Xiao, Y. (2017),
Polarity of points for Gaussian random fields, \emph{Ann. of Probab.} 
\textbf{45}, 4700--4751.

\bibitem{Dodelson} Dodelson, S. (2003), \emph{Modern Cosmology}, Academic
Press

\bibitem{EsIstas10} Estrade, A. and Istas, J. (2010), Ball throwing on
spheres. \textit{Bernoulli} \textbf{16}, 953--970.

\bibitem{Fal90} Falconer, K. J. (1990), \textit{Fractal Geometry --
Mathematical Foundations And Applications.} John Wiley \& Sons Ltd.,
Chichester.

\bibitem{GeHor} Geman, D. and Horowitz (1980), Occupation densities, \emph{\
Ann. Probab.} \textbf{8}, 1--67.

\bibitem{Hawkes} Hawkes, J. (1975), On the potential theory of
subordinators. \emph{Z. Wahrsch. Verw. Gebiete} \textbf{33}, 113--132.

\bibitem{KhoshShi} Khoshnevisan, D. and Shi, Z. (1999), Brownian sheet and
capacity, \emph{Ann. Probab. }\textbf{27}. 1135--1159.

\bibitem{KhoshXiao} Khoshnevisan, D. and Xiao, Y. (2002), Level sets of
additive Levy processes, \emph{Ann. Probab.} \textbf{30}, 62-100.

\bibitem{KratzLeon} Kratz, M. F. and Le\'{o}n, J. R. (2001), Central limit
theorems for level functionals of stationary Gaussian processes and fields, 
\emph{J. Theoret. Probab.} \textbf{14}, no. 3, 639--672.

\bibitem{KratzLeon06} Kratz, M.F. and Le\'{o}n, J. R. (2006), On the second
moment of the number of crossings by a stationary Gaussian process. \emph{\
Ann. Probab.} \textbf{34}, no. 4, 1601--1607.

\bibitem{KratzLeon10} Kratz, M.F. and Le\'{o}n, J. R. (2010), Level curves
crossings and applications for Gaussian models. \emph{Extremes} \textbf{13},
no. 3, 315--351.

\bibitem{LanMarXiao} Lan, X., Marinucci, D. and Xiao, Y. (2018), Strong
local nondeterminism and exact modulus of continuity for spherical Gaussian
fields, \emph{Stoch. Process. Appl.}, to appear.

\bibitem{LanXiao} Lan, X. and Xiao, Y. (2018), H\"{o}lder Conditions of
Local Times and Exact Moduli of non-differentiability for Spherical Gaussian
fields, arxiv:1807.03706.

\bibitem{MPbook} Marinucci, D. and Peccati, G. (2011), \emph{Random Fields
on the Sphere. Representation, Limit Theorem and Cosmological Applications},
Cambridge University Press, Cambridge.

\bibitem{Pitt} Pitt, L. D. (1978), Local times for Gaussian vector fields, 
\emph{Indiana Univ. Math. J.} \textbf{27}, 309--330.

\bibitem{RogTaylor} Rogers, C.A. and Taylor, S.J. (1961), Functions
continuous and singular with respect to a Hausdorff measure, \emph{%
Mathematika} \textbf{8}, 1--31.

\bibitem{schoenberg1942} Schoenberg, I. J. (1942), Positive definite
functions on spheres, \emph{Duke Math. J.} \textbf{9}, 96--108.

\bibitem{Talagrand95} Talagrand, M. (1995), Hausdorff measure of
trajectories of multiparameter fractional Brownian motion, \emph{Ann. Probab.%
} \textbf{23}, 767--775.

\bibitem{Talagrand98} Talagrand, M. (1998), Multiple points of trajectories
of multiparameter fractional Brownian motion, \textit{Probab. Th. Relat.
Fields} \textbf{112}, 545--563.

\bibitem{Taylor} Taylor, J. E. (2006), A Gaussian kinematic formula, \emph{%
Ann. Probab.} \textbf{34}, no. 1, 122--158.

\bibitem{Xiao97} Xiao, Y. (1997), H\"{o}lder conditions for the local times
and the Hausdorff measure of the level sets of Gaussian random fields, \emph{%
Probab. Theory Relat. Fields} \textbf{109}, 129--157.

\bibitem{Xiao09} Xiao, Y. (2009), Sample path properties of anisotropic
Gaussian random fields. In: \emph{A minicourse on Stochastic Partial
Differential Equations,} (D. Khoshnevisan and F. Rassoul-Agha, editors), 
\emph{Lecture notes in Math. }\textbf{1962}, pp. 145-212, Springer, New York.
\end{thebibliography}
\end{document}